\documentclass[final,leqno,onefignum,onetabnum]{siamltex}

\usepackage{amsmath}
\usepackage{amssymb}
\usepackage{mathtools}
\usepackage{formula}
\DeclarePairedDelimiter{\floor}{\lfloor}{\rfloor}

\DeclareMathOperator{\e}{e} 
\DeclareMathOperator{\Jac}{D} 
\renewcommand{\g}{\tilde{f}}

\newcounter{crmk}
\renewcommand\thecrmk{\arabic{crmk}}

\newcounter{cas}
\renewcommand\thecas{\arabic{cas}}

\newtheorem{assumption}{Assumption}
\newtheorem{remark}{Remark}

\title{Stochastic Modeling and Regularity of the Nonlinear Elliptic curl--curl Equation} 

\author{Ulrich R\"omer\footnotemark[2]\ \footnotemark[4] \and Sebastian Sch\"ops\footnotemark[2]\ \footnotemark[3]\ \footnotemark[5] \and Thomas Weiland\footnotemark[2]}

\begin{document}
\maketitle
\newcommand{\slugmaster}{}
\renewcommand{\thefootnote}{\fnsymbol{footnote}}
\footnotetext[2]{Technische Universitaet Darmstadt, Institut f\"ur Theorie Elektromagnetischer Felder, Schlo\ss gartenstrasse 8, D--64289 Darmstadt}
\footnotetext[3]{Technische Universitaet Darmstadt, Graduate School of Computational Engineering, Dolivostraße 15, D--64293 Darmstadt}
\footnotetext[4]{Supported by the "Deutsche Forschungsgemeinschaft" (DFG) under SFB $634$}
\footnotetext[5]{Supported by the "Excellence Initiative" of the German Federal and State Governments and the Graduate School of Computational Engineering at Technische Universitaet Darmstadt, the FP7-ICT-2013-11 Project ”nanoCOPS: Nanoelectronic COupled Problems Solutions” and the "Bundesministerium f\"ur Bildung und Forschung" (BMBF) SIMUROM project.}
\renewcommand{\thefootnote}{\arabic{footnote}}

\begin{abstract}
This paper addresses the nonlinear elliptic curl--curl equation with uncertainties in the material law. It is frequently employed in the numerical evaluation of magnetostatic fields, where the uncertainty is ascribed to the so--called $B$--$H$ curve. A truncated Karhunen--Lo\`{e}ve approximation of the stochastic $B$--$H$ curve is presented and analyzed with regard to monotonicity constraints. A stochastic nonlinear curl--curl formulation is introduced and numerically approximated by a finite element and collocation method in the deterministic and stochastic variable, respectively. The stochastic regularity is analyzed by a higher order sensitivity analysis. It is shown that, unlike to linear and several nonlinear elliptic problems, the solution is not analytic with respect to the random variables and an algebraic decay of the stochastic error is obtained. Numerical results for both the Karhunen--Lo\`{e}ve expansion and the stochastic curl--curl equation are given for illustration.
\end{abstract}



\pagestyle{myheadings}
\thispagestyle{plain}
\markboth{U. R\"omer and S. Sch\"ops and T. Weiland}{Stochastic curl--curl Equation}

\section{Introduction}
Today, it is increasingly acknowledged that uncertainty quantification is an important part within simulation based design. It allows to control the risk of failure as technical devices are designed and operated closer to their physical limits. When the underlying physics are modelled by partial differential equations, uncertain inputs are typically identified with the material's constitutive relation, geometries, initial data or boundary values. In particular the case of random material coefficients for linear equations has received considerable attention in recent years, see \cite{Frauenfelder_2005,xiu2005high,babuska2004galerkin,babuvska2007stochastic} among others. In electromagnetics the coefficients are frequently modeled to be piecewise constant on subdomains. In a stochastic setting, which is adapted here, neglecting anisotropy, on each subdomain the material coefficient can be represented by a single random variable. An important exception is the magnetic properties of ferromagnetic materials, that, in the anhysteretic case, are expressed through a map
\begin{equation}
|\Hr \left(\x \right)| = f \left(\x,|\Br \left(\x \right)| \right),
\label{eq:bh_bh}
\end{equation}
where $\Hr$ and $\Br$ are the magnetic field and flux density, respectively. Magnetic saturation effects, incorporated through the nonlinear dependency on the field magnitude, cannot be neglected in many situations and have been found to be sensitive to uncertainties. In this setting the input randomness is modelled by an infinite--dimensional random field and its discretization must be accomplished. This is complicated by the fact, that each trajectory of the material law has to fulfill smoothness and shape requirements. An example of a material law is given in Figure \ref{fig:bh_f} on the left, showing clearly a monotonic behavior of $f$. An important aspect of this work is the shape--aware modelling of uncertainties in view of a constraint for the derivative
\begin{equation}
0 < \alpha \leq \frac{\partial f \left(\cdot,s \right)}{\partial s} \leq \beta < \infty, \quad s \geq 0.
\label{eq:bh_shapeconstraint}
\end{equation}
To this end, in the literature, closed--form parametric material models have been employed. Among others, we mention different forms of the Brillouin model, see, e.g., \cite{Rosseel_2010, Rama_2012} or the Brauer model \cite{Bartel_2013}. These models are appealing due to their simplicity and a physical interpretation of the parameters can often be given. Finite dimensional random fields are readily obtained by using random- instead of deterministic parameters. However, there is a lack of flexibility due to the a priori fixed dimensionality and specific shape of the analytical functions used. Moreover, the model parameters have been found to be correlated \cite{Rama_2012} and the model might not be used directly in stochastic simulations. The (linear) truncated Karhunen--Lo\`{e}ve expansion \cite{loeve1963probability,ghanem1991stochastic} is known to be a flexible and efficient tool to approximate random fields with high accuracy and to separate stochastic and deterministic variables. It is applied in this paper in view of a stochastic material law with regularity and shape constraints (\ref{eq:bh_shapeconstraint}). The regularity can be controlled by the smoothness of the covariance function whereas the shape constraints imply restrictions on the truncation order, or alternatively, on the uncertainty magnitude. It is observed that this magnitude is dependent on the correlation length of the process, in accordance with \cite[pp. 1281-1283]{babuvska2005solving} in the context of uniform coercivity constraints. Numerical examples with data supported from measurements given in \cite{Rama_2012} support the findings. 

Given random material input data, we discuss a stochastic nonlinear magnetostatic formulation. For related work see \cite{Rosseel_2010,Bartel_2013} and \cite{chauviere2006computational} for the full set of (linear) Maxwell's equations in a stochastic setting. Following a frequently used procedure in the literature \cite{babuvska2005solving,nobile2009analysis}, we first analyze the modelling error arising in magnetic fields through the truncation of the Karhunen--Lo\`{e}ve expansion. Then the problem is reformulated as a high dimensional deterministic one and an approximation scheme is presented. The scheme involves linearization, as well as a finite element and collocation approximation in the deterministic and stochastic variable, respectively. We analyze the stochastic regularity to establish the convergence rate of the numerical procedure. A complication arises here due to the specific type of the nonlinearity. In particular the implicit function theorem cannot be applied and no analytic dependency of the solution with respect to the random variables is obtained. Instead, finite differentiability is established using the chain rule and the deterministic regularity of the solution. Numerical examples will complement the findings. Let us also mention, that the problem considered here is related to many other physical problems, e.g., nonlinear heat conduction. In two dimensions the equations reduce to a nonlinear version of the Poisson equation.

\begin{figure}[!t]
\centering
\includegraphics{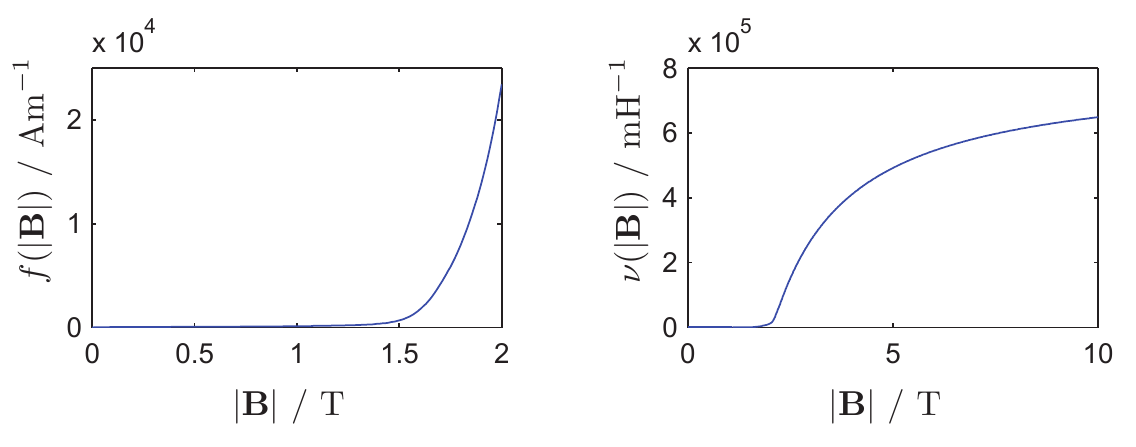}
\caption{Left: example of a nonlinear magnetic material law based on real data. Right: associated magnetic reluctivity, satisfying $\lim_{x \rightarrow \infty} \nu \left(x \right) = \nu_0$.}
\label{fig:bh_f}
\end{figure}

The paper is structured as follows. After a brief description of the magnetostatic model problem in Section \ref{sec:ms}, we introduce randomness in the material law, present the truncated Karhunen--Lo\`{e}ve expansion and analyze the respective truncation error in the stochastic problem in Section \ref{sec:stochastic_ms}. In Section \ref{sec:num_prodecure} we will outline the stochastic collocation method and establish its convergence. This involves a higher order sensitivity analysis w.r.t. the stochastic variables. Finally, in Section \ref{sec:num_examples} convergence results will be illustrated by numerical examples.

\subsection{Notation}
Boldface type is used for vectors $\mf{u}=(u_1,u_2,u_3)$ and vector-functions. Important function spaces are
\begin{equation}
V \defi \H_0 \left(\curl;D \right) = \{\mf{u} \in L^2(D)^3 \ | \ \curl \mf{u} \in L^2(D)^3 \ \mathrm{and} \ \mf{u}\times \mf{n} = 0, \ \mathrm{on} \ \partial D \}, \\
\end{equation}
and for $s>0$ 
\begin{equation}
\H^s \left(\curl,D \right) \defi \{ \ur \in \H^s \left(D \right)^3 \ | \ \curl \ur \in \H^s \left(D \right)^3 \}, 
\end{equation}
see \cite{monk2003finite}. We write $\left(\cdot,\cdot \right)_2$ for the $L^2 \left(D \right)^3$-inner product and $\| \cdot \|_2$ for the associated norm. 

The Euclidean norm is denoted by $|\cdot|$. Also, we introduce the notation $g\D{i} \defi \partial^i_x g \left(x \right)$ for the $i$-th derivative of a function. For $I \subset \R^n$, open and bounded, $\C^k \left(\Ib \right)$ denotes the space of $k$-times differentiable functions with bounded and uniformly continuous derivatives up to order $k$, endowed with the norm 
\begin{equation}
\|g\|_{\C^k \left(\Ib \right)} \defi \max_{0 \leq i \leq k} \sup_{x \in \I} |g\D{i} \left(x \right)|. 
\end{equation}
The sets of real non-negative and positive numbers are denoted $\R_0^+$ and $\R^+$, whereas $\mathbb{N}_0,\mathbb{N}$ refer to the sets of natural numbers with and without zero, respectively. 

We also need a multi--index notation: for $\bs{\gamma} \in \mathbb{N}_0^M$, let $\DJ_{\y}^{\bs{\gamma}} \defi \frac{\partial^{|\bs{\gamma}|_1}}{\partial y_1^{\gamma_1} \dots \partial y_M^{\gamma_M}}$, for $|\bs{\gamma}|_{1} \defi \sum_i \gamma_i >0$ and $\DJ_\y^0$ be the identity operator.
For two multi--indices, $\mi{\gamma}_1 < \mi{\gamma}_2$ holds true if $\mi{\gamma}_1 \leq \mi{\gamma}_2$ holds true component-wise and if $\mi{\gamma}_1 \neq \mi{\gamma}_2$. Also, to a multi--index $\mi{\gamma}$ with $|\mi{\gamma}|_1=n$, we associate a set $\gamma$ with $n$ entries, such that $\partial^n/ \prod_{i \in \gamma} \partial_{y_i}= \DJ_\y^{\mi{\gamma}}$ see \cite[p.518]{motamed2013stochastic}.

\section{The Model Problem}
\label{sec:ms}
We consider the magnetostatic problem on a domain $D$,
\begin{subequations}
\begin{align}
\curl \Hr &= \Jr,& &\mathrm{in} \ D, \\
\div \Br &= 0,& &\mathrm{in} \ D, \\
\Br \cdot \n &= 0,& &\mathrm{on} \ \partial D,
\end{align}
\label{eq:bh_ms_strong}%
\end{subequations}
with outer unit normal $\n$ and divergence free electric current density $\div \Jr = 0$. We introduce the magnetic vector potential $\Ar$, such that $\curl \Ar = \Br$. Then using the material law \eqref{eq:bh_bh}, equations (\ref{eq:bh_ms_strong}) are transformed into a second order $\curl$--$\curl$ problem
\begin{subequations}
\begin{align}
\curl \left (\nu \left(|\curl \Ar| \right) \curl \Ar \right) &= \Jr,& &\mathrm{in} \ D, \label{eq:bh_ms_strong_vector_a}\\
\Ar \times \n &= 0,& &\mathrm{on} \ \partial D, \\
\intertext{supplemented with the Coulomb gauge}
\div \Ar &= 0,& &\mathrm{in} \ D.
\end{align}
\label{eq:bh_ms_strong_vector}%
\end{subequations}
In \eqref{eq:bh_ms_strong_vector_a}, $\nu$ refers to the magnetic reluctivity, defined by 
\begin{equation}
\nu \left(\cdot,s \right) \defi \frac{f \left(\cdot,s \right)}{s}, \quad s>0.
\label{eq:bh_reluctivity}
\end{equation}
Although $\nu$ is the coefficient appearing in the differential equations measurement results are typically given directly for $f$ and thus both will be included in the discussion. Note that, in contrast to $f$, $\nu$ is not necessarily monotonic. From now on we simplify the situation and neglect the spatial dependency in $f$, i.e., $f \left(\x,\cdot \right) = f \left(\cdot \right)$, by abuse of notation. An adaption to the important case of piecewise constant material properties is achieved by minor modifications, see also Remark \ref{rmk:spatial_f}. In absence of hysteresis and anisotropy, the nonlinear magnetic material law can be described, following \cite{Reitzinger_2002,Pechstein_2006}, by a bijective function 
\begin{equation}
f:\R_0^+ \rightarrow \R_0^+: \ |\mf{B}| \mapsto |\mf{H}| = f \left(|\mf{B}| \right).
\label{eq:intro_bh}
\end{equation}
Following \cite{Pechstein_2006}, properties of the deterministic material law $f$ are summarized in the following assumption.
\begin{assumption}\label{as:fdet}
It holds for the deterministic material law that
\begin{subequations}
\begin{align}
&f \mathrm{\ is \ continuously \ differentiable}, \label{as:bh_f1} \\
&0 < \fmin \leq f\D{1} \left(s \right) \leq \fmax < \infty,  \label{as:bh_f2}\\
&f \left(0 \right)=0, \label{as:bh_f3} \\
&\lim_{s \rightarrow \infty} f\D{1} \left(s \right)=\fmax. \label{as:bh_f4}
\end{align}
\end{subequations}
\label{as:bh}%
\end{assumption} 
As $f$ might have increased differentiability properties, we refer to (\ref{as:bh_f1}) as minimal regularity assumption. Note, that $\fmax$ can be identified with the reluctivity of vacuum $\nu_0$. Depending on the problem formulation it might be more convenient to work with the inverse law $f^{-1}$. However, in this case similar assumptions can be made. 

\begin{lemma}
\label{lem:nu}
Let Assumption \ref{as:fdet} be satisfied, then the magnetic reluctivity satisfies for all $s \in \R_0^+$
\begin{subequations}
\begin{align}
& \nu \ \mathrm{is \ continuous \ and } \ \alpha \leq \nu \left(s \right) \leq \beta, \\
&s \mapsto \nu \left(s \right)s \ \mathrm{is \ strongly \ monotone}, \\
&s \mapsto \nu \left(s \right)s \ \mathrm{is \ Lipschitz \ continuous},
\end{align}
\label{eq:bh_prop_reluctivity}%
\end{subequations}
with monotonicity and Lipschitz constants $\fmin,\fmax$, respectively.
\end{lemma}
\begin{proof}
See, e.g., \cite{Pechstein_2004}.
\end{proof}

\subsection{The Nonlinear curl--curl Formulation} \label{sec:problem}
We proceed with the derivation of a weak formulation of the model problem and a result on existence and uniqueness of a solution. Throughout the paper, we consider a bounded, simply connected polyhedral Lipschitz domain $D$. A weak formulation of (\ref{eq:bh_ms_strong_vector}) relies on
\begin{equation}
\V = \{ \ur \in V \ | \ \left(\ur, \grad \varphi \right)_2 = 0, \ \forAll{\varphi}{\W}\},
\end{equation}
the space of functions in $V$ with weak zero divergence. We recall from \cite[Corollary 4.4]{hiptmair2002finite} that the Poincar\'e--Friedrichs--type inequality
\begin{equation}
\|\ur\|_2 \leq C_\mathrm{F} \| \curl \ur \|_2
\label{eq:intro_poincare}
\end{equation}
holds, for all $\ur \in \V$ and $\V$ can be endowed with the norm $\| \ur \|_{\V} := \|\curl \ur\|_2$, see also \cite{amrouche1998vector} for the more general case of multiply connected $D$. Then for $\mf{J}\in L^2 \left(D \right)^3$ the weak formulation reads, find $\mf{A}\in \V$, such that 
\begin{equation}
\int \limits_{D} \nu \left(|\curl \Ar| \right) \curl \Ar \cdot \curl \mf{v} \ \d x = \int \limits_{D} \mf{J} \cdot \mf{v} \ \d x, \quad \forall \mf{v} \in  \V.
\label{eq:apriori_weak}
\end{equation}
Equation (\ref{eq:apriori_weak}) can be written more compactly, by introducing the vector function $\h: \R^3 \rightarrow \R^3,\ \h \left(\mf{r} \right) \defi \nu \left(|\mf{r}| \right) \mf{r}$, as
\begin{equation}
\int \limits_{D} \h \left(\curl \Ar \right) \cdot \curl \mf{v} \ \d x = \int \limits_{D} \mf{J} \cdot \mf{v} \ \d x, \quad \forall \mf{v} \in  \V.
\label{eq:apriori_weak_flux}
\end{equation}
Also, let $\V^*$ denote the dual space of $\V$, by introducing the operator $\A : \V \rightarrow \V^*$ as 
\begin{equation}
\langle \A \ur,\vr \rangle \defi \int \limits_{D} \h \left(\curl \mf{u} \right) \cdot \curl \mf{v} \ \d x, 
\end{equation}
we obtain $\Ar \in \V$ as the solution of
\begin{equation}
\langle \A \Ar,\vr \rangle = \left(\mf{J},\mf{v} \right)_2, \quad \forall \mf{v} \in  \V.
\end{equation}
Existence and uniqueness is guaranteed by the Zarantonello Lemma \cite{zeidler66025nonlinear} as (\ref{eq:bh_prop_reluctivity}) implies
\begin{subequations}
\begin{align}
\langle \A \ur, \ur -\vr \rangle - \langle \A \vr, \ur - \vr \rangle &\geq \fmin \|\ur - \vr\|_{\V}^2, \\
| \langle \A \ur,\mf{w} \rangle - \langle \A \vr, \mf{w} \rangle | &\leq 3 \fmax \| \ur - \vr \|_{\V}  \| \mf{w} \|_{\V},
\end{align}
\label{eq:apriori_monoLip}%
\end{subequations}
see \cite{langer2006coupled}, i.e., the strong monotonicity and Lipschitz continuity of the nonlinear operator $\A$. Moreover, we have the estimate 
\begin{equation}
\|\Ar\|_{\V} \leq \frac{C_{\mathrm{F}}}{\fmin} \|\Jr\|_2.
\label{eq:apriori}
\end{equation}

\begin{remark}\label{rmk:spatial_f}
Typically, for the accurate modelling of magnetic devices an interface problem with several materials, e.g., iron and air, has to be studied. Then, the piecewise defined magnetic reluctivity also satisfies (\ref{eq:bh_prop_reluctivity}) and the problem is still found to be well--posed, see \cite{bachinger2005numerical,Heise_2004}. Also the extension to multiply connected domains would mainly require a modification of the divergence free condition to ensure the norm equivalence \cite{bachinger2005numerical}. In the more general case of $f \left(\x,\cdot \right)$ with arbitrary $x$--dependence, $f \left(\cdot,s \right)$ must additionally be measurable \cite{yousept2013optimal} for all $s$ and consequently $3+1$--dimensional random fields would occur. However, as our focus lies on the nonlinearity in the stochastic setting, for simplicity, we restrict ourselves to simply connected domains with only one homogeneous nonlinear material.
\end{remark}

\section{Uncertainties in the Nonlinear Material Law and Stochastic Formulation}
\label{sec:stochastic_ms}
Randomness is incorporated, as usual, by introducing a probability space $\left(\Omega,\sigmaF,\measP \right)$ and modeling the material law $f$ as a random field $f:\Omega \times \R_0^+ \rightarrow \R_0^+$. A stochastic formulation is based on the following assumption.  
\begin{assumption}\label{as:fstoch}
The stochastic material law $f(\omega,\cdot)$ satisfies Assumption \ref{as:bh} almost surely (a.s.) with constants $\fmin$ and $\fmax$ independent of $\omega$.
\end{assumption}
~\newline
According to Lemma \ref{lem:nu} this implies, that the stochastic reluctivity $\nu:\Omega \times \R_0^+ \rightarrow \R^+$, defined as $\nu \left(\omega,s \right)\defi f \left(\omega,s \right)/s$, for all $s \in \R^+$ satisfies \eqref{eq:bh_prop_reluctivity} a.s., with constants $\fmin,\fmax$ independent of $\omega$.

Setting for $\mf{r} \in \R^3$, $\h \left(\omega,\mf{r} \right)= \nu \left(\omega,|\mf{r}| \right)\mf{r}$, the stochastic curl--curl problem reads a.s. as
\begin{equation}
\int \limits_{D} \h \left(\cdot,\curl \Ar \right) \cdot \curl \mf{v} \ \d x = \int \limits_{D} \mf{J} \cdot \mf{v} \ \d x, \quad \forall{\vr} \in \V.
\label{eq:apriori_weak_stoch}
\end{equation}
By Assumption \ref{as:fstoch} we have a unique solution $\Ar \in L^p (\Omega,\V)$ for all $p \in \mathbb{N}$ by means of (\ref{eq:apriori}).

\subsection{Random Input Discretization by the Truncated Karhunen--Lo\`{e}ve Expansion}
\label{sec:KL}
In the following we restrict ourselves to an open interval $\I \subset \R^+$ and define for $f$ (resp. $\nu$), $\tilde{f} := f|_{\Ib}$. Restricting the uncertainty of the material law to a specific interval is a reasonable assumption in practice and in particular suitable to satisfy the constraints (\ref{as:bh_f3}) and (\ref{as:bh_f4}) in the presence of randomness. A globally defined $f$ can be obtained by a differentiable prolongation from $I$ to $\R^+_0$.

To be used in computer simulations the input random field $\g : \Omega \times \Ib \rightarrow \R^+$ has to be discretized. This is achieved here by introducing the (linear) truncated Karhunen--Lo\`{e}ve expansion 
\begin{equation}
\g_M \left(\omega,s \right)= \E_{\g} \left(s \right) + \sum_{n=1}^{M}  \sqrt{\lambda_n} b_n \left(s \right)Y_n \left(\omega \right).
\label{eq:KL_trunc}
\end{equation}
A stronger dependency of $\g_M$ w.r.t. $Y_n$ might be obtained by performing a Karhunen--Lo\`{e}ve expansion for $\log(\g)$, rather than $f$ \cite{babuvska2007stochastic}. However, a linear expansion as (\ref{eq:KL_trunc}) with a moderate number $M$, might be beneficial in numerical approximations.

We recall the definition of the expected value and covariance function
\begin{align}
\E_{\g} \left(s \right) &\defi \int \limits_{\Omega} \g \left(\omega,s \right) \ \d \measP(\omega), \\
\cov_{\g} \left(s,t \right) &\defi \int \limits_{\Omega} \left(\g \left(\omega,s \right) - \E_{\g} \left(s \right)\right) \left(\g \left(\omega,t \right) - \E_{\g} \left(t \right) \right) \ \d \measP(\omega), \label{eq:bh_meancov}
\end{align}
for $s,t \in \Ib$. Based on the following assumption, some important properties of \eqref{eq:KL_trunc} are briefly recalled here, see, e.g., \cite{loeve1963probability,Frauenfelder_2005,babuvska2005solving,Schwab_2006}.
\begin{assumption}\label{as:KL}
The image of $Y_n, n=1,\dots,M$ is uniformly bounded. Additionally, the covariance satisfies $\cov_{\g} \in \C^l \left(\overline{\I\times \I} \right)$, with $l>2$. 
\end{assumption}
~\newline
Consider the self--adjoint and compact operator $\covop_{\g} : L^2 \left(\I \right) \rightarrow L^2 \left(\I \right)$, given by
\begin{equation}
\left(\covop_{\g} u \right)\left(s_1 \right) \defi \int \limits_{\I} \cov_{\g} \left(s_1,s_2 \right) u \left(s_2 \right) \ \d s_2.
\end{equation}
Then $\left(\lambda_n,b_n \right)_{n=1}^{\infty}$ are eigenpairs of 
\begin{equation}
\covop_{\g} u = \lambda u,
\label{eq:eig_problem}
\end{equation}
where the $\left(b_n \right)_{n=1}^{\infty}$ are orthonormal in $L^2 \left(I \right)$ and $\lambda_1 \geq \lambda_2 \geq \cdots \geq 0$. In (\ref{eq:KL_trunc}), the random variables $\left(Y_n\right)_{n=1}^{\infty}$ subject to 
\begin{equation}
Y_n = \frac{1}{\sqrt{\lambda_n}} \int \limits_{I} \left( \g(\omega,s) - \E_{\g}(s) \right) b_n(s) \d s,
\label{eq:ref}
\end{equation}
for $\lambda_n>0$, are centered and uncorrelated with unit variance. We assume that they are independent and denote their image with $\Gamma_n$ and set $\Gamma := \prod_{n=1}^M \Gamma_n$.

Approximation properties of (\ref{eq:KL_trunc}) are well studied, in particular the $L^2$--error
\begin{equation}
\|\g - \g_M \|_{L^2 \left(\Omega \times I \right)}^2 = \sum_{n=M+1}^{\infty} \lambda_n,
\label{eq:KL_L2}
\end{equation}
is optimal among all $M$--term approximations, see \cite{Schwab_2006}. The remainder in (\ref{eq:KL_L2}) can be bounded, investigating the decay rate of the eigenvalues \cite{Frauenfelder_2005,Schwab_2006}. We obtain
\begin{equation}
0 \leq \lambda_n \leq C n^{-l},
\label{eq:decay}
\end{equation}
with $C>0$, see \cite[Proposition 2.5]{Frauenfelder_2005} and hence,
\begin{equation}
\|\g - \g_M \|_{L^2 \left(\Omega \times I \right)}^2 \leq C_l M^{1-l}.
\end{equation}
An analytic covariance yields an exponential decay. 

Except for several simple covariance functions, e.g., the exponential kernel, the eigenvalue problem (\ref{eq:eig_problem}) has to be solved numerically. This has been addressed, e.g., in \cite{Schwab_2006,oliveira2014spectral}. To assure global differentiability, we employ a Galerkin approximation based on a B--spline space. This leads to a generalized discrete eigenvalue problem. Let $I_{\mathrm{min}}:=\min \left(\Ib \right), \ I_{\mathrm{max}}:= \max \left(\Ib \right)$, and $q,N \in \mathbb{N}$ and consider a (quasi-uniform) sequence 
\begin{equation}
\I_{\mathrm{min}} = s_0 < s_1 < \dots < s_{N} = \I_{\mathrm{max}},
\end{equation}
referred to as mesh $\tau_N$. To $\tau_N$ we associate the B--spline space $\SP^{q,k}_{N}$ of polynomials of degree $q$ on each sub--interval of $\tau_N$ with global continuity $k \in \mathbb{N}$. For the standard iterative construction procedure of B--splines, see, e.g., \cite{de1978practical}. The Galerkin approximation of (\ref{eq:eig_problem}) reads, following \cite[p.111]{Schwab_2006}, find $\left(\lambda_{N,n},b_{N,n} \right)_{n\geq 1}^{\infty} \subset \R \times \SP^{q,k}_N$ subject to
\begin{equation}
\int \limits_{I}\int \limits_{I} b_{N,n} \left(s_1 \right) \cov_{\g} \left(s_1,s_2 \right) v_N \left(s_2 \right)  \d s_1 \d s_2 = \lambda_{N,n} \int \limits_{I} b_{N,n} \left(s \right) v_N \left(s \right) \d s, 
\label{eq:eig_discrete}
\end{equation}
for all $v_N \in \SP^{q,k}_{N}$. By means of the numerically computed eigenpairs, the discrete Karhunen--Lo\`{e}ve expansion reads
\begin{equation}
\g_{M,N}\left(\omega,s \right) = \E_{\g} \left(s \right) + \sum_{n=1}^{M} \sqrt{\lambda_{N,n}} b_{N,n} \left(s \right) Y_{N,n} \left(\omega \right).
\end{equation}
Based on the approximation properties of B-splines, see, e.g., \cite{graham1985iterated}, the Galerkin error contribution can be bounded, following \cite{Schwab_2006}, as
\begin{equation}
\| \g_M - \g_{M,N} \|_{L^2 \left(\Omega \times I \right)}^2  \leq C_M N^{-2 \left(q+1 \right)},
\label{eq:KL_bounds}
\end{equation} 
with $C_M>0$ and $q+1<l$. 


\subsection{Application to the Nonlinear Magnetic Material Law}
In the following we employ a Karhunen--Lo\`{e}ve expansion to discretize the stochastic magnetic material law. We thereby focus on the truncation error, i.e., assume that $N$ is sufficiently large. As in the deterministic case we set $\g\D{i} := \partial_s^i \g \left(\cdot,s \right)$. To ensure the solvability of the PDE the $L^2 \left(\Omega \times I \right)$--convergence of $\g$ to $\g_M$ is not sufficient and we additionally require that $\g_M\D{i}$ satisfies (\ref{eq:bh_shapeconstraint}). 

\begin{lemma}
Let Assumption \ref{as:KL} hold true, then there exists $M_0 \in \mathbb{N}$, such that for $M>M_0$ 
\begin{equation}
\fmin_0 \leq \g\D{1}_M \left(\omega,s \right)  \leq \fmax_0,
\label{eq:nud_unif12}
\end{equation}
holds with positive constants $\fmin_0,\fmax_0$.
\end{lemma}
\begin{proof}
Based on Assumption \ref{as:KL} it follows by \cite[Theorem 2.24]{Schwab_2006} and \eqref{eq:decay}, that for $\lambda_n \neq 0$,
\begin{equation}
\| \g\D{1} \left(\omega,\cdot \right) - \g_M\D{1} \left(\omega,\cdot \right) \|_{L^{\infty} \left(\I \right)} \leq C_{\delta} \sum_{n=M+1}^{\infty} \lambda_n^{1/2-\delta} \leq \underbrace{C_{\delta,l} M^{1- l/2 + \delta l}}_{\defir r_M},
\end{equation}
a.s., with $0 < \delta < 1/2 - 1/l$ and a positive constant $C_{\delta}$ additionally depending on the covariance and $|\Gamma|$. Hence, the sum on the right--hand--side converges for all $M$. Moreover, we can choose $M \geq M_0$ large enough such that for almost all $\omega \in \Omega$, $s \in \I$,
\begin{equation}
\underbrace{\fmin - r_{M_0}}_{\defir \fmin_0} \leq \g\D{1}_M \left(\omega,s \right)  \leq \underbrace{\fmax + r_{M_0}}_{\defir \fmax_0}, \label{eq:nud_unif} 
\end{equation}
with $r_{M_0}\in \left(0,\alpha \right)$. 
\end{proof}
~\newline
Let $f_M$ be a continuously differentiable prolongation of $\g_M$ from $I$ to $\R_0^+$, such that Assumption \ref{as:fstoch} is satisfied. This can be achieved, e.g., by the construction given in \cite{Heise_2004}. Then $\nu_M \left(\omega,s \right) \defi f_M \left(\omega,s \right)/s$ can be used as material coefficient for the stochastic problem. Complementary to this result, formulas assuring that the shape constraint (\ref{eq:nud_unif}) is verified, will be derived in the next section for a concrete setting.
\begin{remark}\label{rmk:normal}
We observe that we cannot model $f$ or $\nu$, as normal (or log--normal) random field, as their derivatives would not be bounded uniformly. In particular the trajectories of $f$ would be non--monotonic with a probability greater than zero.
\end{remark}

\subsubsection{Practical Realization and Numerical Example}
\label{subsec:KL}
\begin{figure}[t!]
\centering
\includegraphics{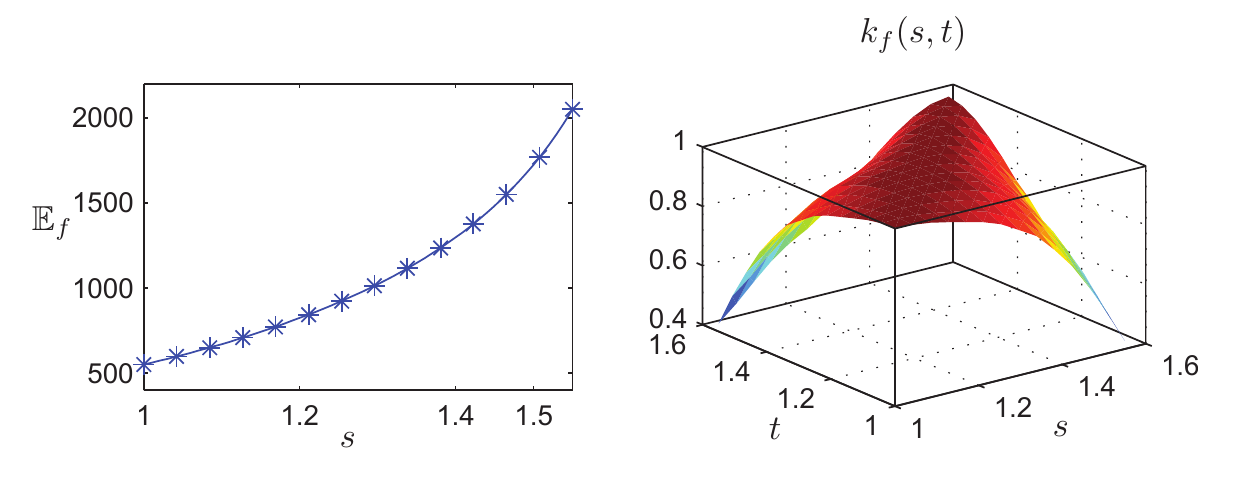}
\caption{Expected value $\E_{f}$ and correlation function $k_{f}$ for the data given in \cite{Rama_2012}.}
\label{fig:NE_ex1}
\end{figure}
In this section a practical realization under minimal assumptions on $f$ is discussed. Measured data is supposed to be available at equidistant points $\I_{\min{}} \leq \hat{s}_1 <\hat{s}_2 < \cdots < \hat{s}_R \leq \I_{\max{}}$. We introduce the table
\begin{equation}
\left\{\left(\hat{s}_i,\hat{f}_{ij} \right), \ i=1,\cdots,R, \ j=1,\cdots,Q\right \}
\end{equation}
and assume that the data is monotonic, i.e., $\hat{f}_{i_1 j} \leq \hat{f}_{i_2 j},$ for $i_1 \leq i_2$ and $j=1,\dots,Q$. The data is interpolated using $\C^1$ monotonicity-preserving cubic splines, see \cite{fritsch1980monotone}. For data with increased oscillations due to measurements a procedure as outlined in \cite{Pechstein_2006,Reitzinger_2002} should be used. 

The situation $Q=1$ is very common in practice, and additional assumptions on the covariance are needed in this case. To solve the Karhunen--Lo\`{e}ve eigenvalue problem, as outlined in Section \ref{sec:KL}, the correlation function
\begin{equation}
k_{\g} \left(s,t \right) = \frac{\cov_{\g} \left(s,t \right)}{\sqrt{\cov_{\g} \left(s,s \right)}\sqrt{\cov_{\g} \left(t,t \right)}},
\label{eq:corr}
\end{equation}
is chosen to be approximated by the Gaussian kernel
\begin{equation}
k_{G} \left(s,t \right) = \delta^2 \e^{-\left( \frac{s-t}{L} \right)^2},
\end{equation}
where $L$ denotes the correlation length and $\delta>0$ a parameter. By rescaling with the interpolated sample variance $\left(\mathrm{Var}_{\hat{f}_i} \right)_{i=1}^R$ we obtain the covariance. The associated discrete eigenpairs $\left(\lambda_n,b_n \right)_{n=1}^N$ are obtained by discretization with $\SP_{N}^{3,1}$. Then, the random material relation is approximated by 
\begin{equation}
\g_M \left(\omega,s \right) = \E_{\g} \left(s \right) + \delta \sum_{n=1}^{M} \sqrt{\lambda_n} b_n \left(s\right) Y_n \left(\omega \right).
\label{eq:NE_KL}
\end{equation}
In (\ref{eq:NE_KL}) $\E_{\g}$ is obtained by projecting the interpolated sample mean $\left(\E_{\hat{f}_i} \right)_{i=1}^{R}$ on $\SP_{N}^{3,1}$. 

Here, we determine $M$ such that the relative information content satisfies
\begin{equation}
\frac{\sum_{n=1}^{M} \lambda_n}{\sum_{n=1}^{M'} \lambda_n} > 0.95,
\end{equation}
where $M' \gg M$. For a more rigorous approach to the recovery of the Karhunen--Lo\`{e}ve approximation from measured data by means of an a posteriori error analysis, we refer to \cite{babuvska2003solving}. Sample realizations of the $Y_n$ can be determined by (\ref{eq:ref}), however here, we model them to be distributed uniformly as $\mathbb{U} \left(-\sqrt{3},\sqrt{3}\right)$. The parameter $\delta$ is used to assure that (\ref{eq:NE_KL}) satisfies the shape constraints (\ref{eq:nud_unif}), see also \cite[p. 1282]{babuvska2005solving} for a related discussion in the context of a linear material coefficient. In particular we only need to assure that $\g_M$ is monotonic. If $\E_{\g}$ can be represented by a spline function, as is the case in the present setting, a simple condition for $\delta$ can be derived to this end. To simplify notation let $\eta_M \left(s\right) =\sum_{n=1}^{M} \sqrt{\lambda_n} b_n(s)$ and $\eta_{M,i}$ be obtained by substituting $b_n$ in the previous relation by its $i$--th spline coefficients. Let $\left(\E_{\g,i}\right)_{i=1}^{N}$ denote the coefficient vector of $\E_{\g} \in \SP_{N}^{3,1}$. Then from \cite{de1978practical} we recall that a sufficient condition for a B--spline to be monotonic is that its coefficients are increasing and hence, monotonicity can be assured by
\begin{equation}
\delta < \min_{i=2,\dots,N}\frac{\E_{\g,i} - \E_{\g,i-1}}{\sqrt{3} |\eta_{M,i}-\eta_{M,i-1}|},
\label{eq:delta}
\end{equation}
where we minimize only over those $i$ with nonzero denominator. In a more general setting one could derive a similar expression by minimizing $\delta_s$ given by
\begin{equation}
\E_{\g}\D{1} \left(s\right) > \delta_s \sqrt{3} |\eta\D{1}_M \left(s\right)|,
\end{equation}
over all $s\in \Ib$.

Let us consider the material uncertainty of an electrical machine. In \cite{Rama_2012}, measured data\footnote{The simulation is based on the original data kindly provided by St\'ephane Cl\'enet.}, representing the material properties from twenty--eight machine stator samples ($Q=28$) from the same production chain was presented. The interval of interest is given by $\Ib = [1,1.55]$ and measurements were taken at $R=14$ equidistant points. For the given criteria we truncate the Karhunen--Lo\`{e}ve expansion with $M=3$. Figure \ref{fig:NE_ex1} depicts both the expected value and the correlation function. For illustration we compare the setting for three different correlation lengths $L=1/20$, $L=1/10$ and $L=1/2$, respectively. Figure \ref{fig:NE_KL} depicts ten sample realizations for each correlation length, where the coefficient $\delta$ is chosen according to (\ref{eq:delta}) and a uniform mesh with $N=60$ spline basis functions is used. It can be readily observed, that a smaller correlation length, corresponding to trajectories with increased oscillations, demands for a smaller $\delta$, i.e., a smaller perturbation magnitude. Also, the largest correlation length $L=1/2$ gives the best agreement with the measured data and will be chosen in what follows.
\begin{figure}[t!]
\centering
\includegraphics{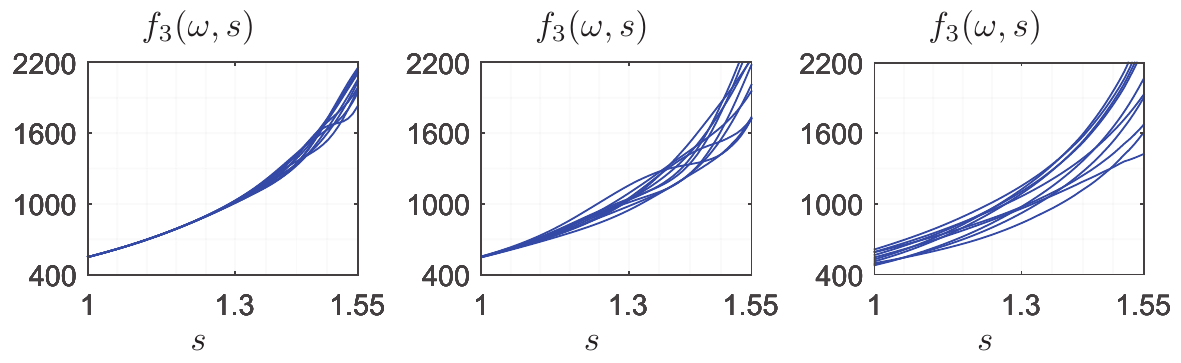}
\caption{Ten sample discretizations for $M=3$, with correlation lengths $L=1/20$, $L=1/10$ and $L=1/2$, respectively. The perturbation amplitudes are $\delta = 0.89$, $\delta=2.10$ and $\delta=2.85$, respectively.}
\label{fig:NE_KL}
\end{figure}

\subsection{Truncation Error and High--Dimensional Deterministic Problem}
We are now going to investigate the modelling error arising from a finite dimensional noise approximation, i.e., when $\nu$ is replaced by $\nu_M$. So far we have explained how this approximation can be achieved by means of the Karhunen--Lo\`{e}ve expansion, however, hereafter, we do not restrict ourselves to this specific case anymore. For given $\nu_M$, for $\r \in \R^3$, let $\h_M \left(\cdot,\r \right) \defi \nu_M \left(\cdot,|\r| \right)\r$ denote the associated vector function. Then $\Ar_M$ is defined a.s. as the solution of 
\begin{equation}
\int \limits_{D} \h_M \left(\cdot,\curl \Ar_M \right) \cdot \curl \mf{v} \ \d x = \int \limits_{D} \mf{J} \cdot \mf{v} \ \d x, \quad \forall{\vr} \in \V.
\label{eq:apriori_weak_stoch_KL}
\end{equation}
 
\begin{proposition}
\label{prp:stability}
Let $\Ar$ and $\Ar_M$ be the solution of (\ref{eq:apriori_weak_stoch}) and (\ref{eq:apriori_weak_stoch_KL}), respectively. Moreover, let $f$ as well as $f_M$ satisfy Assumption \ref{as:fstoch}, with constants $\fmin,\fmin_0$, respectively. Then we have a.s.
\begin{equation}
 \| \Ar - \Ar_M \|_{\V} \leq \|\nu - \nu_M\|_{L^{\infty} \left(\mathbb{R}^+ \right)} \frac{C_{\mathrm{F}} \|\mf{J}\|_2}{\fmin \fmin_0}.
\label{eq:KL_error_moments}
\end{equation}
\end{proposition}
\begin{proof}
        As $f$ satisfies Assumption \ref{as:fstoch}, we have a uniform strong monotonicity property, i.e., a.s.
        \begin{equation}
           \fmin \|\Ar - \Ar_M \|_{\V}^2 \leq \int \limits_{D} \left(\h \left(\cdot,\curl \Ar \right) - \h \left(\cdot,\curl \Ar_M \right) \right) \cdot \curl \left(\Ar - \Ar_M \right) \ \d x
        \end{equation}							
        and because of equations (\ref{eq:apriori_weak_stoch}) and (\ref{eq:apriori_weak_stoch_KL}) and the Cauchy--Schwarz inequality
        \begin{equation}
            \fmin \|\Ar - \Ar_M \|_{\V} \leq  \|\h_M \left(\cdot,\curl \Ar_M \right) -\h \left(\cdot,\curl \Ar_M \right) \|_2.
        \end{equation}
        For the right--hand--side we further obtain
				\begin{align}
				 \|\h_M (\cdot,&\curl \Ar_M) -\h \left(\cdot,\curl \Ar_M \right) \|_2  \\
        &\leq \left( \int \limits_{D} \left( \left(\nu_M \left(\cdot,|\curl \Ar_M| \right) -\nu \left(\cdot,|\curl \Ar_M| \right) \right) \curl \Ar_M \right)^2 \ \d x\right)^{1/2}  \\
				& \leq \|\nu_M - \nu\|_{L^{\infty} \left(\mathbb{R}^+ \right)} \|\Ar_M\|_{\V}.
        \end{align}
				The result follows from $\| \Ar_M\|_{\V} \leq \frac{C_{\mathrm{F}} \|\mf{J}\|_2}{\fmin_0}$.				
\end{proof}
~\newline
Due to (\ref{eq:KL_error_moments}) we have control of the truncation error. For simplicity, this error is omitted in the following, i.e., we assume that the uncertain input has a finite dimensional noise representation: 
\begin{assumption}\label{as:FN}
The random field $\nu$, resp. $f$, depends (continuously) on $M$ independent random variables solely, i.e., a.s.
\begin{equation}
\nu \left(\Y \left(\omega \right),s \right) = \nu \left(\omega,s \right),
\label{eq:apriori_fdn1}
\end{equation}
where $\Y= \left(Y_1,Y_2,\dots,Y_M \right)$.
\end{assumption}

We recall that $\Y$ may also refer to random variables in closed--form representations of $\nu$ or coefficients in spline models, among others. Based on Proposition \ref{prp:stability} and Assumption \ref{as:FN} it follows from the Doob--Dynkin Lemma \cite{rao2006probability} (cf. \cite{babuvska2002solving}) that we can write 
\begin{equation}
\Ar \left(\Y \left(\omega \right),\x \right) = \Ar \left(\omega,\x \right).
\end{equation}
We recall that the $Y_n$ have a bounded image and a joint probability density function
\begin{equation}
\rho: \Gamma \rightarrow \R^+,
\end{equation}
such that $\rho \left(\Y \right) = \rho_1 \left(Y_1 \right) \rho_2 \left(Y_2 \right) \dots \rho_M \left(Y_M \right)$. For all random variables $X \in L^1_\rho \left(\Gamma \right)$, such that $X \left(\omega \right)= X \left(\Y \left(\omega \right) \right)$ we introduce 
\begin{equation}
\E \left[X \right] = \int \limits_{\Gamma} X \left(\y \right) \rho \left(\y \right) \ \d y.
\end{equation}
We now introduce the following assumption for $f: \Gamma \times \R_0^+ \rightarrow \R_0^+$.
\begin{assumption}\label{as:fparam}
The stochastic material law $f(\y,\cdot)$ satisfies Assumption \ref{as:bh} for $\rho$--almost all $\y \in \Gamma$, with constants $\fmin$ and $\fmax$ independent of $\y$.
\end{assumption}
~\newline
The stochastic problem can be recast into a deterministic one with $3+M$ dimensions. To this end, let $\Hspace \defi L^2_{\rho} \left(\Gamma \right) \otimes \V$ be the closure of formal sums $\mathcal{\ur} = \sum_{i=1}^n v_i \mathbf{w}_i$, where $\{v_i\}_{i=1,n} \subset L^2_{\rho} \left(\Gamma \right)$ and $\{\mathbf{w}_i\}_{i=1,n} \subset \V$, with respect to the inner product 
\begin{equation}
(\ur,\ur)_{\Hspace} = \E \left [(\ur,\ur)_{\V} \right ],
\end{equation}
cf. \cite{babuvska2005solving}.
Let $\h:\Gamma \times \R^3 \rightarrow \R^3$. We seek $\Ar \in \Hspace$ such that
\begin{equation}
\int \limits_{D} \h \left(\y,\curl \Ar \left(\y \right) \right) \cdot \curl \vr \ \d x = \int \limits_{D} \mf{J} \cdot \vr \ \d x, \quad \forAll{\vr}{\V},
\label{eq:apriori_strong_det}
\end{equation}
where $\Ar \left(\y \right) \defi \Ar \left(\y,\cdot \right)$.

\section{A Stochastic Collocation Method for the Nonlinear curl--curl Formulation}
\label{sec:num_prodecure}
In the following, the variables $\x \in D$ and $\y \in \Gamma$ will be referred to as deterministic and stochastic variable, respectively. The solution of (\ref{eq:apriori_strong_det}) requires discretization in both variables as well as a linearization procedure. To this end, we carry out:
\begin{romannum}
	\item
	\emph{deterministic} discretization based on lowest order $\H \left(\curl \right)$--conforming finite elements with maximum stepsize $h$,
	\item
	\emph{stochastic} discretization based on a collocation procedure on a tensor grid or sparse grid of level $q$,
	\item
	$l$--times iteration of the \emph{linearized} system of equations by means of the Ka\u{c}anov or Newton--Raphson method.
\end{romannum}
As we will see in Section \ref{sec:regularity}, the use of global, higher order polynomials over $\Gamma$ is justified by the regularity of the solution, whereas the collocation procedure is particularly attractive for nonlinear problems due to the ease of implementation. We will then proceed by analyzing the approximation error originating from finite element discretization $\errorh$, stochastic collocation $\errorp$ and linearization $\errorl$, respectively. By the triangle inequality these errors can be decomposed as
\begin{equation}
\|\Ar - \Ar_{h,q,l} \|_{\Hspace} \leq \underbrace{ \| \Ar - \Ar_{h} \|_{\Hspace}}_{\defir \errorh} + \underbrace{\| \Ar_{h} - \Ar_{h,q} \|_{\Hspace}}_{\defir \errorp}  + \underbrace{\| \Ar_{h,q} - \Ar_{h,q,l} \|_{\Hspace}}_{\defir \errorl}.
\end{equation}
Additional sources of error can be identified, in particular quadrature errors and the error from numerically solving linear systems of equations. However, these errors will be omitted here. We also claim that all three steps of the proposed scheme commute. This has been shown for the deterministic case in \cite{logg2012automated} and generalizing to the stochastic collocation method is straightforward.  

\subsection{Galerkin Finite Element Approximation}
Equation (\ref{eq:apriori_strong_det}) is approximated in the deterministic variable by the Galerkin finite element method. Higher order schemes are well established and could be employed, however, as our focus lies on the stochastic part, we restrict ourselves to lowest order schemes. We consider discretizations of the Lipschitz polyhedron $D$ with a simplicial mesh $\mesh_h$, with maximum size $h>0$. The mesh is assumed to be quasi--uniform in the sense of \cite[Definition 4.4.13]{brenner2008mathematical}, in particular 
\begin{equation}
\min_{T \in \mesh_h} \mathrm{diam} \left(B_T \right) \geq C_D h,
\label{eq:apriori_quasiuniform}
\end{equation}
where $B_T$ denotes the largest ball contained in $T$. We then introduce the discrete spaces
\begin{align}
V_h :=& \{\ur \in V \ | \ \ur|_{T} = \mf{a}_T + \mf{b}_T \times \x, \quad \mf{a}_T,\mf{b}_T \in \R^3, \quad \forall T \in \mesh_h\}, \\
W_h :=& \{ v \in \H_0^1 \left(D \right) \ | \ v|_{T} = \mf{a}_T \cdot \x + b_T, \quad \mf{a}_T \in \R^3 ,b_T \in \R, \quad \forall T \in \mesh_h\},
\end{align}
i.e., $V_h$ and $W_{h}$ are spanned by lowest order N\'ed\'elec and Lagrange elements, respectively. As in the continuous case, the space of (discrete) divergence free functions, is introduced as
\begin{equation}
\V_h := \{\ur_h \in V_h \ | \ \left(\ur_h,\grad v_h \right)_2 = 0,  \quad \forall v_h \in W_{h}\}.
\end{equation}
We observe that $\V_h$ is not a subspace of $\V$, as (weak) discrete divergence free functions are not (weak) divergence free in general. The deterministic finite element approximation consists in computing $\Ar_{h}: \Gamma \rightarrow  \V_h$ such that for $\rho$--almost all $\y \in \Gamma$,
\begin{equation}
\int \limits_{D} \h \left(\y,\curl\Ar_h \left(\y \right) \right) \cdot \curl \vr_h \ \d x = \int \limits_{D} \mf{J} \cdot \vr_h \ \d x, \quad \forAll{\vr_h}{\V_h},
\label{eq:apriori_fem}
\end{equation}   
holds. Existence and uniqueness can be established based on a discrete Poincar\'e--Friedrichs inequality \cite[Theorem 4.7]{hiptmair2002finite}, see, e.g., \cite{yousept2013optimal}.

\subsection{Stochastic Collocation Method}
Stochastic discretization is based on a collocation approach using either a tensor or a sparse grid, see, e.g., \cite{xiu2005high,babuvska2007stochastic,babuvska2010stochastic,motamed2013stochastic}. Starting with the (isotropic) tensor grid, following \cite{motamed2013stochastic}, the collocation points are given as
\begin{equation}
H_{q,M}^{\mathrm{T}} \defi \{ y_1^1,y_1^2,\dots,y_1^{n(q)} \} \times \{ y_2^1,y_2^2,\dots,y_2^{n(q)} \} \times \dots \times \{y_M^1,y_M^2,\dots,y_M^{n(q)} \},
\end{equation}
where in each dimension $m=1,\dots,M$, we have $n(q) = p(q) + 1$ collocation points and $N_q = n(q)^M$ in total. Note that $p$ refers to the underlying polynomial degree, which we identify with the level for the tensor grid case as $p(q)=q$. Also, a global index $k$ is associated to the local indices in the usual way \cite{babuvska2010stochastic}. The collocation points are chosen as the roots of the orthogonal polynomials associated to the probability density function $\rho$. As commonly done \cite{nobile2009analysis,babuvska2010stochastic} we introduce the notation $\y = \left(y_m,\yrest{m} \right), \ \yrest{m}= \left(y_1,\dots,y_{m-1},y_{m+1},\dots,y_M \right)$. Let $\Q{p}\left(\Gamma_m \right) $ be the space of polynomials of degree at most $p$ in $\Gamma_m$. Then we introduce in each dimension the one--dimensional Lagrange interpolation operator $\mathrm{I}_{p}^m : \C \left(\Gamma_m;V \right) \rightarrow \Q{p} \left(\Gamma_m \right) \otimes V$ such that 
\begin{equation}
\mathrm{I}_{p}^m \ur \left(\y \right) = \sum_{i=1}^{p+1} \ur \left(y_m^i,\yrest{m} \right) l_{m}^i \left(y_m \right),
\end{equation}
where $l_{m}^i \left(y_m \right)$ is the Lagrange polynomial of degree $p$ associated to the point $y_m^i$. The tensor grid interpolation formula reads as
\begin{equation}
\mathcal{I}_{q,M} \ur \left(\y \right) = \mathrm{I}_{p(q)}^1 \otimes \dots \otimes \mathrm{I}_{p(q)}^M \ur \left(\y \right) = \sum_{k=1}^{N_q} \ur \left(\y_k \right) l_k \left(\y \right),
\label{eq:int_tens}
\end{equation}
where $l_k \left(\y \right)$ is the global Lagrange polynomial associated to the point $\y_k \in H_{q,M}^{\mathrm{T}}$. 

An isotropic tensor grid, with $n$ points in each direction, can only be used for moderate dimensions $M$, as the total number of collocation points grows as $n^M$. Therefore, collocation in higher dimensions is based on sparse grids \cite{barthelmann2000high,nobile2008sparse}. For simplicity we consider isotropic Smolyak grids, solely. Anisotropic sparse grids are discussed, e.g., in \cite{nobile2008anisotropic}. Following \cite{motamed2013stochastic}, let $\mf{j} \in \mathbb{N}_0^M$ be a multi-index and 
\begin{equation}
\mathcal{I}_{\mf{j},M} \ur \left(\y \right) = \sum_{i_1=1}^{p(j_1)+1} \dots \sum_{i_M=1}^{p(j_M)+1} \ur \left( \left(y_{1,j_1}^{i_1},\dots,y_{M,j_M}^{i_M} \right) \right) \prod_{m=1}^M l_{m,j_m}^{i_m} \left(y_m \right),
\end{equation}
the associated multi-dimensional Lagrange interpolation operator, where the Gau{\ss} knots $\{ y_{m,j_m}^{i_m} \}_{i_m=1}^{p(j_m)+1}$ and the Lagrange polynomials $\{l_{m,j_m}^{i_m}\}_{i_m=1}^{p(j_m)+1}$ now also depend on the multi-index $\mf{j}$. For the choice $p(j)=2^{j}$ for $j>0$ and $p(0) = 0$, the Smolyak formula is given by
\begin{equation}
\mathcal{A}_{q,M} \ur \left(\y \right) = \sum_{q-M+1 \leq |\mf{j}|_1 \leq q} (-1)^{q-|\mf{j}|_1} \binom{M-1}{q-|\mf{j}|_1} \mathcal{I}_{\mf{j},M} \ur \left(\y \right).
\label{eq:int_sparse}
\end{equation}
The associated sparse grid is denoted $H_{q,M}^{\mathrm{S}}$. Evaluating \eqref{eq:int_tens} and \eqref{eq:int_sparse} requires solving 
\begin{equation}
\int \limits_{D} \h \left(\y_k,\curl \Ar_h \left(\y_k \right) \right) \cdot \curl \vr_h \ \d x = \int \limits_{D} \mf{J} \cdot \vr_h \ \d x, \quad \forAll{\vr_h}{\V_h},
\label{eq:apriori_col}
\end{equation}
for all collocation points $\y_k$ in $H_{q,M}^{\mathrm{T}}$ and $H_{q,M}^{\mathrm{S}}$, respectively.

\subsection{Linearization}
At each collocation point, iterative linearization is carried out until the linearization error is found to be sufficiently small. The $l$--th iterate, $l \in \mathbb{N}$, consists in computing $\Ar_{h,q,l} \in \Q{q}\left(\Gamma \right) \otimes \V_h$ such that for $k=1,\dots,N_q$
\begin{equation}
\int \limits_{D} \h_{\L} \left(\y_k,\curl \Ar_{h,q,l}(\y_k)\right) \cdot \curl \vr_h\ \d x = \int \limits_{D} \mf{J} \cdot \vr_h \ \d x, \ \forAll{\vr_h}{\V_h},
\label{eq:apriori_fem_col_lin}
\end{equation}   
where $\Q{q}\left(\Gamma \right)$ refers to the polynomial space associated either to tensor or to Smolyak interpolation. For a precise definition of these spaces, see, e.g., \cite{babuvska2010stochastic}. The representation \eqref{eq:apriori_fem_col_lin} follows \cite{el2011guaranteed,chaillou2006computable} and in particular we consider the linearization, 
\begin{equation}
\h_{\L} \left(\cdot,\mf{r}\right) = \nu \left(\cdot,|\mf{r}_{l-1}|\right) \mf{r},
\label{eq:apriori_fluxlin}
\end{equation}
for $\mf{r},\mf{r}_{l-1} \in \R^3$. This is usually referred to as Ka\u{c}anov method or successive substitution in the literature. Note that the case of the Newton--Raphson method, i.e., 
\begin{equation}
\h_{\L} \left(\cdot,\mf{r}\right) =  \nu \left(\cdot,|\mf{r}_{l-1}|\right) \mf{r} + \frac{\nu\D{1} \left(\cdot,|\mf{r}_{l-1}|\right)}{|\mf{r}_{l-1}|} \mf{r}_{l-1} \otimes \mf{r}_{l-1} \left(\mf{r} - \mf{r}_{l-1}\right), 
\end{equation}
is also covered. Under restrictions on the starting point $\Ar_{h,q,0}$ and damping, if necessary, $\Ar_{h,q,l}$ converges to $\Ar_{h,q}$. At each step equation (\ref{eq:apriori_fem_col_lin}) is well--posed by the Lax--Milgram Lemma for both choices. For the Ka\u{c}anov method, this follows by observing that 
\begin{equation*}
{\nu \left(\y,|\curl \Ar_{l-1}|\right) \in [\fmin,\fmax]}. 
\end{equation*}
For the Newton--Raphson method we refer to Lemma \ref{lem:nud} below.

\subsection{Stochastic Regularity and Convergence Analysis}
\label{sec:regularity}
Convergence of the stochastic collocation method introduced above, can be established once the regularity of the solution is known. Whereas for the present nonlinear elliptic problem, the regularity w.r.t. the deterministic variable $\x$ is well known, to our knowledge, the stochastic regularity of the solution $\Ar$ w.r.t. $\y$ has not been investigated. In the case of a linear elliptic PDE it is well known, that under some mild assumptions the solution is an analytic function of the stochastic variable \cite{babuvska2010stochastic,nobile2009analysis,cohen2011analytic}. Similar results hold true for several types of nonlinear problems, see \cite{chkifa2014breaking,chkifa2014high}. Here, the mapping 
\begin{equation}
\y \mapsto \nu \left(\cdot,|\curl \Ar \left(\y\right)|\right)
\label{eq:apriori_complex}
\end{equation} 
is real, but not complex differentiable, see \cite{jack1990methods,bachinger2005numerical}, and this impedes a complex analysis. Moreover, the techniques presented in \cite{chkifa2014breaking}, based on the implicit function theorem, cannot be applied as a norm gap arises. More precisely, the nonlinearity $\h : L^{p}(D)^3 \rightarrow L^q(D)^3$ can be differentiated only for $q<p$, see \cite{wachsmuth2014differentiability}. Higher order differentiability even requires a larger difference between $p$ and $q$. Therefore, we conduct an explicit higher order sensitivity analysis to precisely determine the stochastic regularity. 

We define $\nu_{\bs{\gamma}}\left(\y,|\r|\right) \defi \DJ_{\y}^{\bs{\gamma}}\nu \left(\y,|\r|\right)$, $\h_{\mi{\gamma}}(\y,\r) \defi \DJ_{\y}^{\bs{\gamma}} \h(\y,\r)$ and $\Ar_{\bs{\gamma}} \defi \DJ_{\y}^{\bs{\gamma}} \Ar \left(\y \right)$, respectively. 
We will encounter derivatives of the function $\h \left(\cdot,\r\right) = \nu \left(\cdot,|\r|\right) \r$ with respect to $\r$, i.e., multi--linear maps $\Jac_\r^k \h \left(\cdot,\r\right): \R^{3k}  \rightarrow \R^3$. Of particular interest is the Jacobian $\Jac_\r^1 \h$, identified with the differential reluctivity tensor as
\begin{equation}
\nud \left(\y,\r\right) \defi \left \{
\begin{aligned}
&\nu \left(\y,|\r|\right) + \frac{\nu\D{1}\left(\y,|\r|\right)}{|\r|}  \r \otimes \r, \ \r \neq 0, \\
&\nu (\y,0), \ \r = 0.
\end{aligned}
\right .
\end{equation}
An important property is stated in the following Lemma. 
\begin{lemma}
\label{lem:nud}
Let Assumption \ref{as:fparam} hold true. Then the differential reluctivity tensor satisfies
\begin{subequations}
\begin{align}
|\nud \left(\y,\mf{s}\right)| &\leq \fmax_\d, \label{eq:reluctivity_diff1} \\
\mf{r}^\top  \nud \left(\y,\mf{s}\right)  \mf{r} &\geq \fmin_\d |\mf{r}|^2, \label{eq:reluctivity_diff2}
\end{align}
\label{eq:reluctivity_diff}
\end{subequations}
for $\fmin_\d,\fmax_\d>0$ and all $\y \in \R^M, \mf{r},\mf{s} \in \R^3$. 
\end{lemma}
\begin{proof}
This result has been established, e.g., in \cite[Lemma 3.1]{langer2006coupled}.
\end{proof}
~\newline
In particular this implies, that the bilinear form $\bb_\d \left(\ur;\cdot,\cdot\right)$, defined by
\begin{equation}
\bb_\d \left(\ur;\vr,\mf{w}\right) \defi \int \limits_D \nud \left(\cdot,\curl \ur\right) \curl \vr \cdot \curl \mf{w} \d \x,
\end{equation}
is continuous and coercive on $\V$. We recall that $\V_h$ is not a subspace of $\V$. However, $\bb_\d$ is continuous and coercive on $\V_h$, too. The form $\bb_\d$ arises naturally when sensitivities are computed and also in the linearized system obtained by the Newton--Raphson method, which is well--posed at each iteration step by the properties just established. 

Provided that $f$ is $k$--times differentiable, the question arises whether higher order derivatives of the solution exist as well. In a first step, we consider the case $k \leq 3$. We impose the following assumption on the material law.  
\begin{assumption}\label{as:fsens}
For the parametric material law there holds
\begin{equation}
f \in \C^{3}(\Gamma \times \R^+) 
\label{as:bh_par_f1}
\end{equation}
with bounded and uniformly continuous derivatives, and additionally \[f\D{2}(\cdot,0) = f\D{3}(\cdot,0) =0.\]
\end{assumption} 

The vanishing higher order derivatives of $f$ around the origin are used here to ensure differentiability in presence of the absolute value. Under the previous assumption we infer that $\Jac_\r^k \h$ is bounded.
\begin{lemma}
\label{lem:multilinear}
Let Assumption \ref{as:fsens} hold true. Then for $|\mi{\alpha}|_1 \leq k \leq 3$, $\partial_\r^{\mi{\alpha}} h_j$ is continuous and $|\partial_\r^{\mi{\alpha}} h_j| \leq C_k$.
\end{lemma}
\begin{proof}
For $\mf{r} \neq 0$, as $h_j(\cdot,\mf{r}) = f(\cdot,|\mf{r}|) r_j / |\mf{r}|$ we see that $h_j$ is $k$-times continuously differentiable. Hence, if $\partial_{\mf{r}}^{\mi{\alpha}} h_j(\cdot,\mf{r})$ is bounded for $\mf{r} \rightarrow \infty$ and $\r \rightarrow 0$ the result follows. For $\r \rightarrow \infty$ we observe that $f\D{k}$ is bounded (by Assumption \ref{as:fsens}) and that the same holds true for $\partial_{\mf{r}}^{\mi{\alpha}} (\r/|\r|)$. 

For $\r \rightarrow 0$ we first observe that by the rule of l'H\^{o}pital
\begin{equation}
\nu\D{k-1}(\cdot,0) = f\D{k}(\cdot,0)/k
\end{equation} 
and hence $\nu\D{1}(\cdot,0) = \nu\D{2}(\cdot,0) = 0$ by Assumption \ref{as:fsens}. 
Hence, using the expressions for $\Jac_\r^{k} \h$ given in Appendix \ref{ap:A} we infer that $\Jac_{\mf{r}}^1 \h (\cdot, 0) (\s_1) = \nu(\cdot,0) \s_1$ and that $\Jac_{\mf{r}}^2 \h (\cdot, 0) = \Jac_{\mf{r}}^3 \h (\cdot, 0) = 0$.
\end{proof}
~\newline 
By examining the proof of Lemma \ref{lem:multilinear} we observe that if $|\mi{\beta}|_1 \leq k$, $\Jac_\r^{k - |\mi{\beta}|_1} \h_{\mi{\beta}}$ is bounded, too. 

In the continuous case, formally differentiating the strong form of the boundary value problem, we obtain for the derivative $\Ar_{\mi{\gamma}}$ 
\begin{subequations}
\begin{align}
\curl \left(\nud \left(\cdot,\curl \Ar \right) \curl \Ar_{\mi{\gamma}} \right) &= \curl \mf{F}_k \left( \left(\Ar_{\mi{\alpha}} \right)_{\mi{\alpha}<\mi{\gamma}} \right), \\
\div \Ar_{\mi{\gamma}} &= 0, \\
\Ar_{\mi{\gamma}} \times \n &= 0,
\end{align}
\label{eq:sens_equation}%
\end{subequations}
where, as shown in the Appendix, $\mf{F}_k$ is given by 
\begin{equation}
\mf{F}_k = -\sum_{{0 \leq \mi{\alpha} \leq \mi{\gamma} }} \!\! \binom{\mi{\gamma}}{\mi{\alpha}} \!\!
\sum_{{ \pi \in \Pi^*}} \Jac_\r^{\mathrm{ca}\left(\pi \right)}\h_{\mi{\alpha}} \left(\cdot,\curl \Ar\right)  \left(\curl  \Ar_{\mi{\pi}_1},\dots, \curl \Ar_{\mi{\pi}_{\mathrm{ca}(\pi)}} \right).
\label{eq:sens_rhs}
\end{equation}
With $\Pi^*$ we denote the set of partitions of $\gamma -\alpha$, such that $\mathrm{ca}\left(\pi \right) > 1$ if $\alpha=0$, where $\mathrm{ca}\left(\pi \right)$ refers to the cardinality of $\pi$. We refer to Appendix \ref{ap:B} for a more detailed definition of the underlying sets of $\mf{F}_k$. We observe, that $\mi{\pi}_i < \mi{\gamma}$ for $i=1,\dots,\mathrm{ca}\left(\pi \right)$ and hence the derivatives contained in the right--hand--side are of lower order. Equation (\ref{eq:sens_equation}) is the basis for establishing the regularity of the solution with respect to the stochastic variable. This in turn determines the convergence rate of the corresponding discretization error. Before bounding the collocation error, we discuss the finite element error. To this end, we assume the following:
\begin{assumption}\label{as:space_reg}
The solution of (\ref{eq:apriori_strong_det}) has the additional regularity $\Ar \in L^{\infty} \left(\Gamma,\H^s \left(\curl;D\right)\right)$, with $s \in (1/2,1]$, and the same holds true for the weak solution of (\ref{eq:sens_equation}), if it exists.
\end{assumption} 
~\newline
Based on this regularity assumption, we can establish the following result:
\begin{lemma}
\label{lem:fem}
Let Assumptions \ref{as:fparam} and \ref{as:space_reg} hold true, then the deterministic error is bounded as 
\begin{equation}
\errorh \leq C h^{s},
\end{equation}
where $C$ depends on $\mf{A}$ and on $s$ but is independent of $\y$.
\begin{proof}
For the deterministic error $\errorh$, also in the present nonlinear case, C\'ea's Lemma 
\begin{equation}
\|\Ar \left(\y\right) - \Ar_{h} \left(\y\right) \|_{\V} \leq C_1 \inf_{\vr_h \in \V_h} \|\Ar \left(\y\right) - \vr_h \|_{\V},
\end{equation}
holds for $\rho$--almost all $\y \in \Gamma$, see \cite{bachinger2005numerical}. Then from \cite[Theorem 5.41]{monk2003finite} we obtain
\begin{equation}
\|\Ar \left(\y\right) - \Pi_h \Ar \left(\y\right) \|_{\V} \leq C_2 h^{s} \|\Ar \left(\y\right)\|_{\H^s \left(\curl;D\right)},
\end{equation}
where $\Pi_h$ is the canonical interpolation operator \cite[p. 134]{monk2003finite} and the constant $C_2$ depends on $s$, but is independent of $\y$. As $\Ar  \in L^{\infty} \left(\Gamma,\H^s \left(\curl;D\right)\right)$, we conclude that $\errorh \leq C_3 h^s$.
\end{proof}
\end{lemma}
Based on the sensitivity analysis carried out in this section, we can now use Theorems 4 and 5 of \cite{motamed2013stochastic} to establish the main result.
\begin{theorem}
\label{thm:conv}
Let Assumptions \ref{as:fparam}, \ref{as:fsens} and \ref{as:space_reg} hold true. There holds 
\begin{equation}
\|\Ar - \Ar_{h,q}\|_{\Hspace} \leq C h^{s} + \errorp,
\end{equation}
where $C$ is the constant from Lemma \ref{lem:fem}. Moreover, let $i_s$ be an integer $1 \leq i_s \leq k$, such that $i_s = 1$ for $s \in (1/2,3/4)$, $i_s=2$ for $s \in [3/4,1)$ and $i_s=3$ for $s=1$, respectively. For the isotropic tensor grid collocation method we have 
\begin{equation}
\errorp \leq \left \{ 
\begin{aligned}
&C_{\mathrm{T}} q^{-i_s}, \\
&\frac{C_{\mathrm{T}}}{2} N_q^{-i_s/M},
\end{aligned}
\right .
\end{equation}
with respect to the level (polynomial degree) $q$ and the number of collocation points $N_q$, respectively. For the sparse grid collocation method and $M \leq i_s$, there holds
\begin{equation}
\errorp \leq \left \{ 
\begin{aligned}
&C_{\mathrm{S}} (q+1)^{2M} 2^{-\floor*{i_s/M}(q+1)}, \\
&C_{\mathrm{S}} \left(1+\log_2 \left(\frac{N_q}{M} \right) \right)^{2M} N_q^{-\floor*{i_s/M}\frac{\log 2}{\xi + \log M}},
\end{aligned}
\right .
\end{equation}
with respect to the level $q$ and the number of collocation points $N_q$ in the sparse grid, respectively. The constants $C_{\mathrm{T}},C_{\mathrm{S}}$ depend on $s,\mf{A},\nu,\rho$ and $\xi \approx 2.1$.
\end{theorem}
\begin{proof}
The deterministic error estimate has been established in Lemma \ref{lem:fem}. In a first step, we bound the stochastic collocation error for the isotropic tensor grid. In this case, the collocation error can be recast as an interpolation error
\begin{equation}
\Ar_h - \Ar_{h,q} = \Ar_h - \mathrm{I}_{q} \Ar_h.
\end{equation} 
We will consider the case $M=1$, solely, as the result for $M > 1$ follows by induction, see, e.g., \cite[Lemma 7.1]{canuto2007fictitious} or \cite[Theorem 4]{motamed2013stochastic}. The collocation error is related to the best--approximation error in $\Q{q} \left(\Gamma\right) \otimes \V_h$, following \cite{babuvska2010stochastic}, as
\begin{equation}
\|\Ar_h - \mathrm{I}_{q} \Ar_h \|_{L^2_{\rho} \left(\Gamma \right) \otimes \V} \leq C_1 \inf_{\vr \in \Q{q} \left(\Gamma \right) \otimes \V_h} \|\Ar_h - \vr \|_{L^{\infty} \left(\Gamma,\V \right)}.
\end{equation}
The error decay then depends on the smoothness of $\Ar$ (and hence $\Ar_h$) with respect to the stochastic variable. We note that, as $M=1$, $\partial_{\y}^{\mi{\gamma}} \Ar$ simplifies to $\partial_y^{i_s} \Ar$. Using (\ref{eq:sens_equation}) and the coercivity of $\bb_\d$ we can formally bound
\begin{equation}
\|\partial_y^{i_s} \Ar\|_{\V} \leq C_2 \|\mf{F}_{i_s}\|_2.
\end{equation}
Applying the generalized H\"older inequality and Lemma \ref{lem:multilinear} yields
\begin{align*}
\|\mf{F}_{i_s}\|_2 &\leq C_3 \max_\pi \| \ |\partial_{\y}^{\mi{\pi}_1} \curl  \Ar| \cdots |\partial_{\y}^{\mi{\pi}_{\mathrm{ca}(\pi)}} \curl \Ar| \ \|_2 \\
&\leq C_3 \max_\pi \| \partial_{\y}^{\mi{\pi}_1} \curl  \Ar \|_{L^p(D)^3} \cdots \| \partial_{\y}^{\mi{\pi}_{\mathrm{ca}(\pi)}} \curl \Ar \|_{L^p(D)^3},
\end{align*}
where $p/2 = \mathrm{ca}(\pi) \leq i_s$ and $C_3$ depends on $\nu$ and $s$. Hence, we have to choose $i_s$ such that $ \partial_{y}^{j}\curl \Ar(y) \in L^{2i_s}(D)^3$, for $j=0,\dots,i_s-1$ and $\rho$-almost all $y \in \Gamma$:
\begin{romannum}
\item
For $s \in (1/2,3/4)$ we obtain $i_s=1$ as we only have $\curl \Ar(y) \in L^2(D)^3$. 
\item For $s \in [3/4,1)$ we can set $i_s=2$ as $\partial_y^j \curl \Ar(y) \in L^4(D)^3$ holds, for $j=0,1$.  
\item[] This follows from the Sobolev embedding theorem, see, e.g. \cite[Theorem 3.7]{monk2003finite}, 
\item[] as we have $\partial_y^j \curl \Ar(y) \in \H^s(D)^3$.
\item 
For $s=1$ we obtain $i_s=3$, as we can show that $\partial_y^j \curl \Ar(y) \in L^6(D)^3$, for $j=0,1,2$ 
\item[] using again the Sobolev embedding theorem.
\end{romannum}
Hence, the results of Jackson quoted from \cite{motamed2013stochastic} yield
\begin{equation}
\|\Ar_h - \mathrm{I}_{q} \Ar_h \|_{L^2_{\rho}\left(\Gamma \right) \otimes \V} \leq C_4 q^{-i_s} \max_{j=0,\dots,i_s}\|\partial_y^j \Ar_h \|_{L^{\infty}\left(\Gamma,\V \right)}
\label{eq:apriori_Jackson}
\end{equation}
with $C_4$ depending on $s$ and by induction for $M>1$
\begin{equation}
\errorp \leq C_5 q^{-i_s} \sum_{m=1}^M \max_{j=0,\dots,i_s}\|\partial_{y_m}^j \Ar_h \|_{L^{\infty}\left(\Gamma,\V \right)} \leq C_6 q^{-i_s}. 
\end{equation}

The results for the sparse grid collocation error can be inferred from \cite{motamed2013stochastic} or \cite{barthelmann2000high} once bounds on mixed derivatives of order $k$, i.e., $\partial_{\y}^{\mi{\gamma}} \Ar$ with $\gamma_j \leq k$ for $j=1,\dots,M$ have been established. Using the same arguments as above we obtain
\begin{equation}
\|\partial_{\y}^{\mi{\gamma}} \Ar\|_{\V} \leq C_7 \|\mf{F}_{|\mi{\gamma}|_1}\|_2
\end{equation}
for $|\mi{\gamma}|_1 \leq i_s$. This in turn ensures bounded mixed derivatives of order $k=\floor*{i_s/M}$ for the case $M \leq i_s$, solely. Then the results follow from Theorem 5 of \cite{motamed2013stochastic}.
\end{proof}
\begin{remark}\label{rmk:twod}
Concerning the stochastic discretization error, in the two--dimensional case we can choose $i_s = 2$ for $s \in [1/2,2/3)$ and $i_s = 3$ for $s \in [2/3,1]$, respectively. 
\end{remark}
\begin{remark}\label{rmk:lin}
The linearization error could be included into this convergence estimate: provided that the initial values $\Ar_{h,q,0}\left(\y_k \right)$ are sufficiently close to $\Ar_{h,q} \left(\y_k \right)$, there exists $r \in (0,1)$, such that
\begin{equation}
\errorl \leq C r^l
\end{equation}
for the linearization error of the solution $\Ar_{h,q,l}$ of (\ref{eq:apriori_fem_col_lin}) obtained by the Ka\u{c}anov method. An improved estimate could be obtained for the Newton--Raphson method.
\end{remark}
\begin{remark}
The sparse grid convergence rate of $\mathcal{O}(N_q^{-\gamma \floor*{i_s/M}})$, with $\gamma \in (0,1)$, is smaller than the tensor grid convergence rate of $\mathcal{O}(N_q^{-i_s/M})$. This reduction, due to a lack of mixed regularity of the solution, is also observed in the numerical experiments in Section \ref{sec:num_examples}. However, for smooth solutions the sparse grid approach is expected to be more efficient for large $M$.
\end{remark}

We now address the question whether a fast convergence, e.g., a rate of $q^{-k}$ for the tensor grid and arbitrary $k \in \mathbb{N}$, can be obtained under suitable regularity assumptions on the material input data. This is true if $\curl \Ar$ is bounded uniformly, e.g., for smooth domains and data, as can be seen by the preceding arguments. Also, if we accept a non--uniform constant with respect to the mesh size $h$, the decay of the stochastic discretization error can be improved, as stated in the following.
\begin{proposition}
\label{prp:main}
Let Assumption \ref{as:fparam} hold true and let $\partial_\r^{\mi{\alpha}} \partial_\y^{\mi{\beta}} \h$ be continuous and bounded, for $|\mi{\alpha}|_1 + |\mi{\beta}|_1 \leq k \in \mathbb{N}$. Then we have
\begin{equation}
\errorp \leq C_{\mathrm{T},h,k} \ q^{-k}, 
\label{eq:conv}
\end{equation}
for a tensor grid and
\begin{equation}
\errorp \leq C_{\mathrm{S},h,k} (q+1)^{2M} 2^{-\floor*{k/M}(q+1)}, 
\label{eq:conv_sparse}
\end{equation}
if $M \leq k$, for a sparse grid, respectively. The constants additionally depends on $\mf{A},\nu,\rho$.
\end{proposition}
\begin{proof}
As in the proof of Theorem \ref{thm:conv} we bound $\mf{F}_k$ as
\begin{align}
\left\|\mf{F}_k \right\|_2 &\leq C_k \max_\pi \left\| \ |\curl \DJ_\y^{\mi{\pi}_1} \Ar_h| \cdots |\curl \DJ_\y^{\mi{\pi}_{\mathrm{ca}(\pi)}} \Ar_h| \ \right\|_2 \\
&\leq \tilde{C}_k h^{-k/2} \max_\pi \| \DJ_\y^{\mi{\pi}_1} \Ar_h \|_{\V} \cdots \|\DJ_\y^{\mi{\pi}_{\mathrm{ca}(\pi)}} \Ar_h \|_{\V},
\end{align}
where we have used
\begin{equation}
\|\curl \DJ_\y^{\mi{\pi}_i} \Ar_h\|_{L^{\infty} \left(D \right)^3} \leq h^{-1/2} \|\curl \Ar_h \|_{2},
\label{eq:apriori_discrete_norm_equivalence}
\end{equation}
by the shape--uniformity of the mesh (\ref{eq:apriori_quasiuniform}). This yields 
\begin{equation}
\|\partial_\y^{\mi{\gamma}} \Ar_h \|_{\V} \leq C_{k,h} \| \Ar_h \|_{\V},
\label{eq:apriori_higher_order_estimate}
\end{equation}
with $|\mi{\gamma}|_1=k$, for the solution $\Ar_h \left(\y \right)$ of (\ref{eq:apriori_fem}). The result can now be established along the lines of the proof of Theorem \ref{thm:conv} by observing that (\ref{eq:apriori_Jackson}) holds true for $k$ arbitrary.
\end{proof}
~\newline
A similar idea was applied in \cite{Pechstein_2004} in a two--dimensional setting to establish the Lipschitz continuity of the Newton operator.

\section{Numerical Examples}
\label{sec:num_examples}
Several numerical examples are presented here to illustrate the findings. In $2$D we consider a stochastic version of the $p$--Laplacian, where the solution is known exactly, and the L--shaped domain, respectively. In $3$D the magnetostatic TEAM benchmark problem $13$ \cite{team} will be discussed with the stochastic material law of Section \ref{subsec:KL}. Here, numerical results are obtained using FEniCS \cite{logg2012automated}, and all non--uniform meshes are generated by Gmsh \cite{geuzaine2009gmsh}. Tensor and sparse grids are generated using the Sparse grids Matlab kit \cite{back.nobile.eal:comparison,sg}.

In $2$D, the curl operator reduces to $\curl \left(u \ \mf{e}_3 \right) = \left(\partial_{x_2} u,-\partial_{x_1} u, 0 \right)$, where $\mf{e}_3$ denotes the unit vector in the third dimension. The $2$D equivalent of (\ref{eq:apriori_weak}) then reads, find $u \in \H_0^1 \left(D \right)$ subject to 
\begin{equation}
\left(\nu \left(|\grad u| \right) \grad u,\grad v \right)_2 = \left(J,v \right)_2 \quad \forall{v} \in \H_0^1 \left(D \right).
\label{eq:ne_Dirichlet}
\end{equation} 
Equation (\ref{eq:ne_Dirichlet}) is approximated by means of $W_h$, i.e., lowest order nodal finite elements.

\subsection{p--Laplace}
\begin{figure}[!t]
\centering
\includegraphics[width=0.95\textwidth]{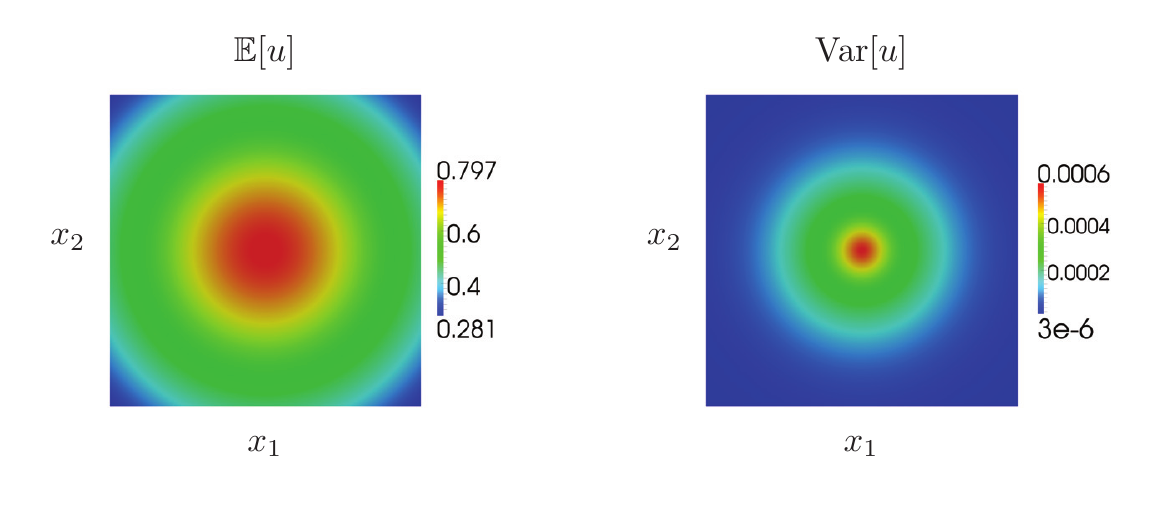}
\caption{$p$--Laplace example. Left: $\E_u(x_1,x_2)$ of analytic solution. Right: $\var_u(x_1,x_2)$ of analytic solution.}
\label{fig:pLaplace1}
\end{figure}
\begin{figure}[!t]
\centering
\includegraphics{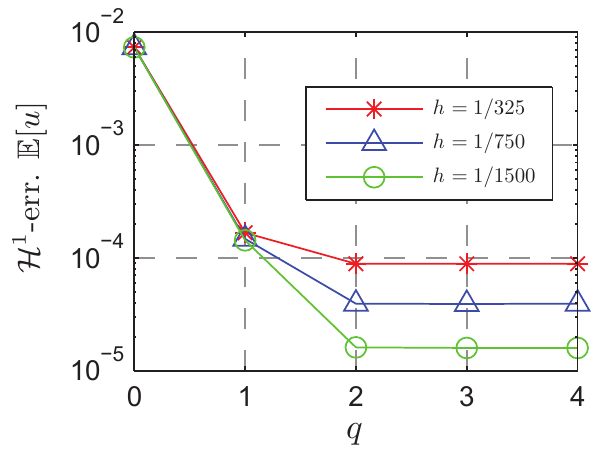}
\caption{$p$--Laplace example. Discretization error, i.e., $\H^1$--error of the expected value for different levels of spatial discretization and small linearization error. A fast error convergence is observed until the spatial discretization error level is attained. As a reference, $\E[u]$ is approximated with $\E_q[u]$, $q=10$.}
\label{fig:pLaplace2}
\end{figure}
We consider $D = \left(0,1 \right)\times \left(0,1 \right)$, a constant current $J=2$ and for $s \in \R^+$ 
\begin{equation}
\nu \left(s \right) = s^{p-2},
\label{eq:NE_reluctivity_pLaplace}
\end{equation} 
giving rise to the $p$--Laplace problem considered in \cite{el2011guaranteed}. The solution is given by 
\begin{equation}
u \left(p,\left(x_1,x_2 \right)\right) = -\frac{p-1}{p}|\left(x_1,x_2\right) - \left(0.5,0.5\right)|^{\frac{p}{p-1}} + \frac{p-1}{p} 0.5^{\frac{p}{p-1}},
\label{eq:NE_plaplcae_sol}
\end{equation}
as we use the inhomogeneous Dirichlet boundary condition $u |_{\partial D}$. Modeling $Y=p$ as a single random parameter, with $Y>1$ a.s., we obtain a stochastic problem. For each $Y$ we have $u(Y,\cdot) \in \H^1(D)$ and even $\grad u(Y,\cdot) \in L^{\infty}(D)^2$. Hence, we expect a fast spectral stochastic convergence. Note that (\ref{eq:NE_reluctivity_pLaplace}) violates the assumptions on the reluctivity for $s \rightarrow \{0,\infty\}$. However, monotonicity and continuity results can be obtained as outlined in \cite{chaillou2007posteriori}. We model $Y$ as uniformly distributed on $\left(3,5 \right)$, i.e., $Y \sim \mathbb{U}\left(3,5 \right)$. Using a uniform triangulation, (\ref{eq:ne_Dirichlet}) is iteratively solved by means of the Ka\u{c}anov method (with damping) until the solution increment is below $10^{-12}$ in the discrete $L^{\infty}$--norm, which yields a negligible linearization error compared to the other sources of error. In Figure \ref{fig:pLaplace1} we depict the expected value and variance of the solution, respectively. In Figure \ref{fig:pLaplace2} the error $\| \E[u - u_{q,h}] \|_{\H_0^1 \left(D \right)}$ is depicted for different values of $h$. As we are not aware of a closed form solution of $\E[u]$, we approximate it by $\E_q[u]$, with $q=10$. From Figure \ref{fig:pLaplace2} an exponential decay of the stochastic error can be observed until the corresponding discretization error level is attained. 
 
\subsection{L--Shaped Domain}
\begin{figure}[!t]
\centering
\includegraphics{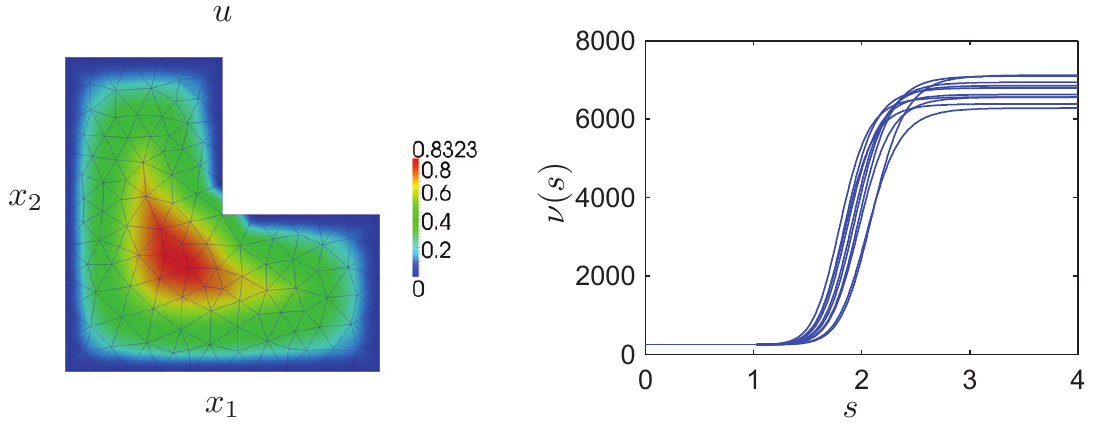}
\caption{L--shaped domain. Left: solution evaluated with unperturbed material law and constant source current density $J=10^5$ on the coarsest grid FE grid. Right: random realizations of stochastic material model (\ref{eq:Cimrak}).}
\label{fig:lshape1}
\end{figure}
Another two--dimensional example is the L--shaped domain, given by $[-1,1]^2  \backslash [0,1]^2$. Homogeneous Dirichlet boundary conditions are applied together with the constant excitation $J=10^5$. We adopt the material model 
\begin{equation}
\nu(s) = d + \frac{c s^{2 b}}{a^b + s^{2 b}}
\label{eq:Cimrak}
\end{equation}
from \cite{cimrak2012material}. Given initial values $a_0=1.78$, $b_0=14$, $c_0=6000$ and $d_0=245$, we set $a = a_0(1 + 0.2 Y_1)$, $b = b_0$, $c = c(1+ 0.2 Y_2)$, and $d=d_0$ to introduce randomness, where $Y_{1,2} \sim \mathbb{U}\left(-\sqrt{3},\sqrt{3} \right)$. In Figure \ref{fig:lshape1} on the left and right we depict the solution for $Y_1=Y_2=0$ and random realizations of the reluctivity, respectively. Linearization is carried out as in the previous example. A coarse finite element (FE) grid as depicted in Figure \ref{fig:lshape1} on the left, referred to as FE grid $1$, is uniformly refined two (FE grid $2$) and four times (FE grid $3$), respectively. The stochastic error for a tensor grid $H_{q,2}^{\mathrm{T}}$, w.r.t. both the polynomial degree and the number of grid points, is shown in Figure \ref{fig:lshape2}. As a reference, a higher order polynomial approximation ($q=9$) in the stochastic variable is used on each grid. For this example there holds $\grad u \in \H^s(D)^2$, with $s = 2/3 - \epsilon$, with $\epsilon>0$. According to Theorem \ref{thm:conv} in $2$D, see Remark \ref{rmk:twod}, at least a decay of $q^{-2}$ is expected as $s \in [1/2,2/3)$. For the finest FE grid we numerically observe a decay of even $q^{-2.987}$. Hence, the convergence rate seems to be insensitive to the small $\epsilon$ parameter and the result could possibly be slightly improved. For grids $1$ and $2$ the decay is also faster than predicted. In Figure \ref{fig:lshape2_sparse} the errors for a sparse $H_{q,2}^{\mathrm{S}}$ are depicted. As $\floor*{i_s/M}=1$ in this case, the convergence is assured. We observe an algebraic decay w.r.t. to the number of grid points. Also the tensor grid approach is more efficient in this case as, e.g., for $N_q = 49$ and FE grid 3, the error using $H_{q,2}^{\mathrm{S}}$ is $4.12 \times 10^{-4}$, whereas the error using $H_{q,2}^{\mathrm{T}}$ is $1.77 \times 10^{-5}$. 
\begin{figure}[!t] 
\centering
\includegraphics{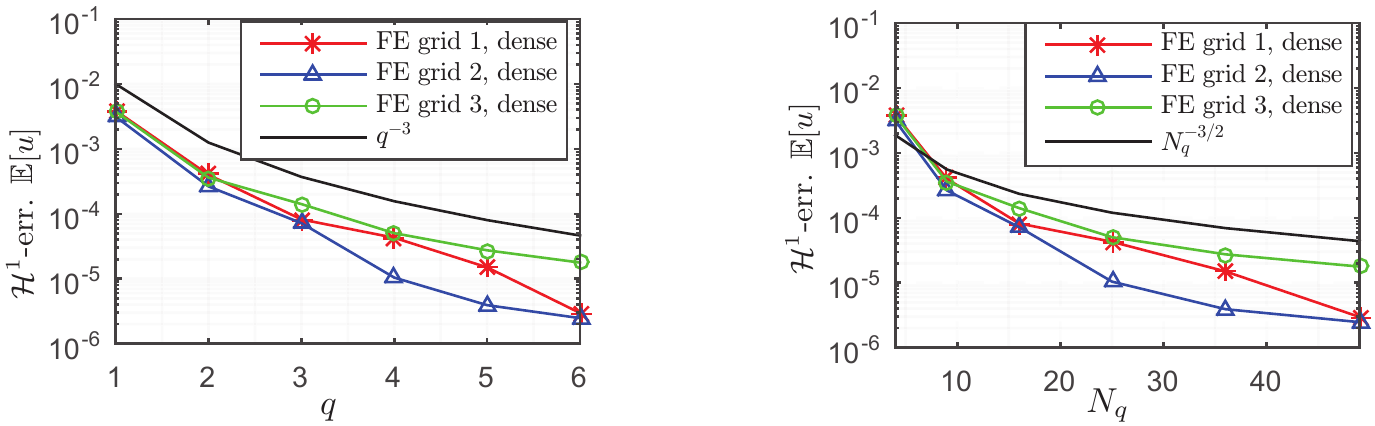}
\caption{L--shaped domain. Estimated stochastic error for tensor grid $H_{q,2}^{\mathrm{T}}$ on three different FE grids. The reference solution is computed with polynomial degree $q=9$. Left: error w.r.t. underlying polynomial degree. Right: error w.r.t. number of grid points.}
\label{fig:lshape2}
\end{figure}
\begin{figure}[!t] 
\centering
\includegraphics{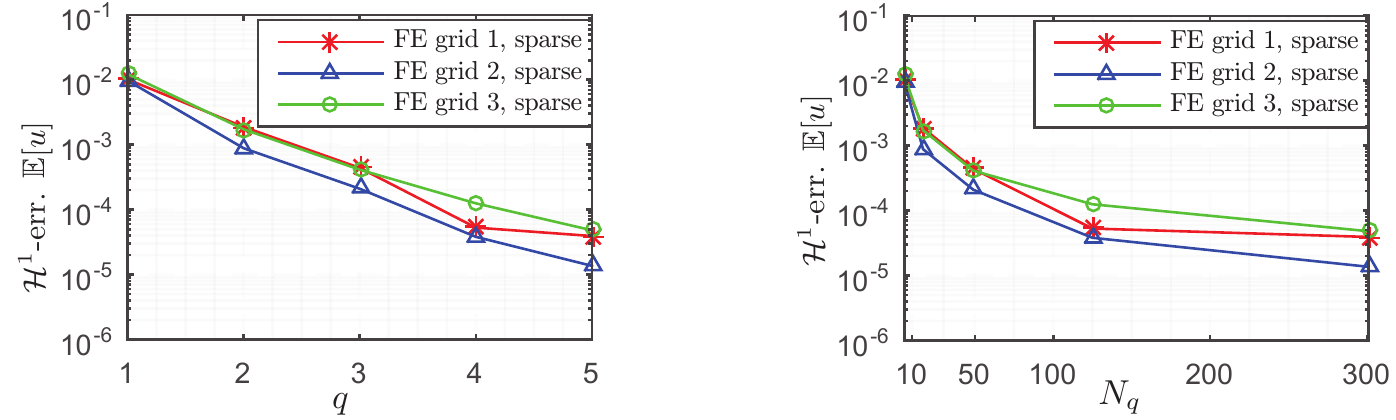}
\caption{L--shaped domain. Estimated stochastic error for sparse Smolyak grid $H_{q,2}^{\mathrm{S}}$ on three different FE grids. The reference solution is computed on a dense grid with polynomial degree $q=9$. Left: error w.r.t. sparse grid level. Right: error w.r.t. number of grid points.}
\label{fig:lshape2_sparse}
\end{figure}

\subsection{TEAM Benchmark}
TEAM benchmarks are setup to validate electromagnetic codes and in particular magnetic field simulations. Here, we investigate the nonlinear magnetostatic TEAM $13$ problem. A magnetic field in three thin iron sheets is generated by a rectangular coil, with blended corners, as depicted in Figure \ref{fig:team2} on the right. Due to symmetry, only the upper half of the iron sheets is visualized. In a pre--processing step, an electrokinetic problem is solved to obtain the source current distribution $\Jr$ with a total imposed current of $3000 \mathrm{A}$ per cross section. Gauging is enforced through a Lagrange multiplier and a mixed formulation see, e.g., \cite{monk1994superconvergence}. The computational domain is truncated, applying homogeneous Dirichlet boundary conditions at a boundary, sufficiently far away from the problem setup. Strictly speaking this iron--air interface problem would require minor modifications of the analysis presented in this paper, as mentioned in Remark \ref{rmk:spatial_f}. We replaced the material properties of the original benchmark and employed the stochastic $B$--$H$ curve presented in Section \ref{subsec:KL}, with $L=1/2$ and $\delta=2$ instead. Extrapolation beyond the data range is carried out as described in \cite{Reitzinger_2002}. Note that for this case we are concerned with $\C^1$ trajectories, solely. A tetrahedral mesh, see also Figure \ref{fig:team2}, is generated using the software Gmsh \cite{geuzaine2009gmsh} and two steps of uniform refinement are carried out. We refer to the different FE grids with $7388$, $41958$ and $261196$ total degrees of freedom as grid $1$, $2$ and $3$, respectively. In Figure \ref{fig:team2} on the left the magnetic flux density distribution is depicted for the expected value of the $B$--$H$ curve.     
\begin{figure}[!t]
\centering
\includegraphics[width=0.95\textwidth]{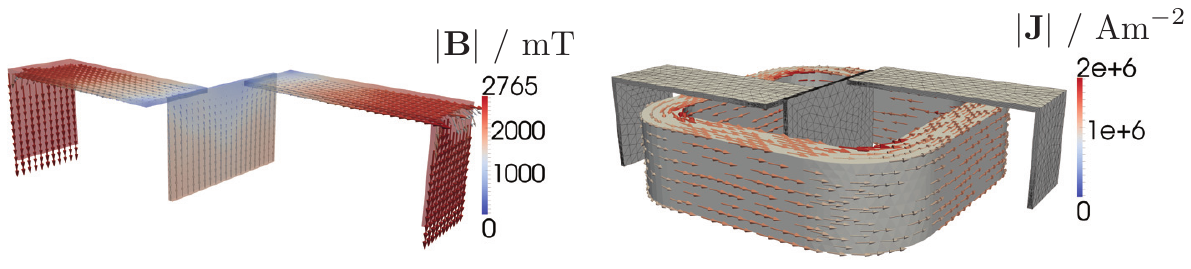}
\caption{TEAM $13$ problem. Left: magnetic flux density distribution in nonlinear material region for expected value of stochastic $B$--$H$ curve. Right: source current distribution within coil and coarsest mesh of nonlinear material region.}
\label{fig:team2}
\end{figure}
The nonlinear problem is linearized and iterated as explained in the $p$--Laplace example with a linearization increment below $10^{-5}$ in the discrete $L^{\infty}$--norm. In view of the expected low regularity of the solution and the small number of random inputs a tensor grid $H_{q,3}^{\mathrm{T}}$ is used. The stochastic convergence is depicted in Figure \ref{fig:team3} for all three FE grids. Here, the stochastic error is estimated as $\H(\curl)$--norm of $\E[\Ar_{h,q,l}]-\E[\Ar_{h,q+1,l}]$, as we do not have an analytical solution. The initially rapid convergence deteriorates until the linearization error level is attained. We also observe that the convergence is algebraic. However, the convergence is faster than the predicted rate $q^{-1}$ in view of the limited differentiability of the reluctivity.
\begin{figure}[!t]
\centering
\includegraphics{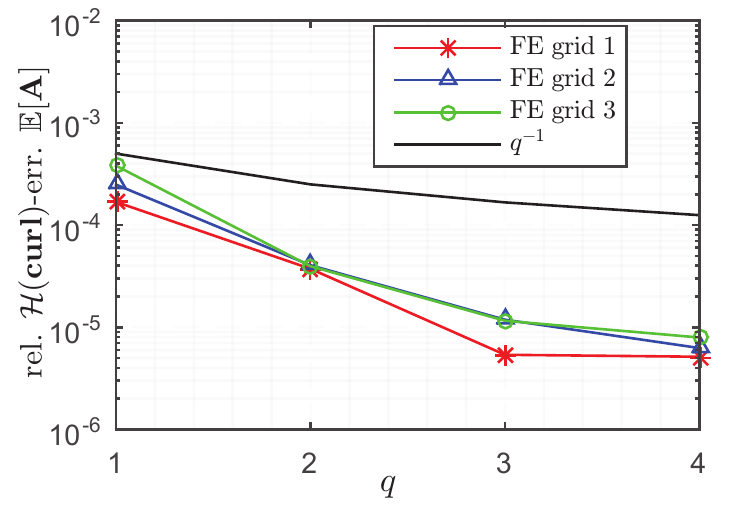}
\caption{TEAM $13$ problem. Estimated stochastic discretization error using $H_{q,3}^{\mathrm{T}}$, i.e., $\H(\curl)$--error of the expected value for different FE grids. Convergence is limited to the linearization error level.}
\label{fig:team3}
\end{figure}

\section{Conclusions}
In this work, we have addressed the stochastic nonlinear elliptic curl--curl equation with uncertainties in the material law. Assumptions on the input have been formulated in order to obtain a well--posed stochastic formulation and it was shown that they can be fulfilled when a suitable discretization of the random input is carried out by the truncated Karhunen--Lo\`{e}ve expansion. As monotonicity is required for the trajectories of the material law, oscillations can only occur with a rather small magnitude. A stability result for approximations of the random input was also derived. In the second part of the paper, a stochastic collocation method for the nonlinear curl--curl equation was analyzed. Under moderate differentiability assumptions on the material law, a convergence rate of $q^{-k}$ was obtained for the stochastic collocation error using tensor grids, where $1 \leq k \leq 3$. For smooth boundaries and data this estimate holds true for all $k \in \mathbb{N}$. These estimates were shown to be in good agreement with numerical results for academic and benchmark examples. Convergence results for sparse grids were also obtained. However, in this case convergence can be expected only for a limited number of random input parameters. 

\begin{appendix}

\section{Tensors}
\label{ap:A}
The tensors $\Jac_{\mf{r}}^k \h$ for $1 \leq k \leq 3$, $\s_1,\s_2,\s_3 \in \R^3$ and $\mf{r} \neq 0$ read as
\begin{equation}
\Jac_{\mf{r}}^1 \h (\cdot, \mf{r}) (\s_1) =   \frac{\nu\D{1}\left(\cdot,|\r|\right)}{|\r|}  \r (\r \cdot \s_1) + \nu \left(\cdot,|\r|\right) \s_1, 
\end{equation}
\begin{multline}
\Jac_{\mf{r}}^2 \h (\cdot, \mf{r}) (\s_1,\s_2) = \left(\frac{\nu\D{2}\left(\cdot,|\r|\right)}{|\r|^2} - \frac{\nu\D{1}\left(\cdot,|\r|\right)}{|\r|^3}\right)  \r (\r \cdot \s_1) (\r \cdot \s_2) \\
	+ \frac{\nu\D{1}\left(\cdot,|\r|\right)}{|\r|} \left ( (\r \cdot \s_1) \s_2 + (\r \cdot \s_2) \s_1 + \r (\s_1 \cdot \s_2) \right ),
\end{multline}
and
\begin{multline}
\Jac_{\mf{r}}^3 \h (\cdot, \mf{r}) (\s_1,\s_2,\s_3) = 
\left(\frac{\nu\D{3}\left(\cdot,|\r|\right)}{|\r|^3} - 3\frac{\nu\D{2}\left(\cdot,|\r|\right)}{|\r|^4} + 3\frac{\nu\D{1}\left(\cdot,|\r|\right)}{|\r|^5}\right)  \r (\r \cdot \s_1) (\r \cdot \s_2) (\r \cdot \s_3) \\
+\left(\frac{\nu\D{2}\left(\cdot,|\r|\right)}{|\r|^2} - \frac{\nu\D{1}\left(\cdot,|\r|\right)}{|\r|^3}\right)  \left(\s_3 (\r \cdot \s_1) (\r \cdot \s_2) + \r (\s_3 \cdot \s_1) (\r \cdot \s_2)  + \r (\r \cdot \s_1) (\s_3 \cdot \s_2) \right) \\
+\left(\frac{\nu\D{2}\left(\cdot,|\r|\right)}{|\r|^2}  - \frac{\nu\D{1}\left(\cdot,|\r|\right)}{|\r|^3} \right)\left ( (\r \cdot \s_1) \s_2 + (\r \cdot \s_2) \s_1 + \r (\s_1 \cdot \s_2) \right ) (\r \cdot \s_3) \\
+\frac{\nu\D{1}\left(\cdot,|\r|\right)}{|\r|} \left ( (\s_3 \cdot \s_1) \s_2 + (\s_3 \cdot \s_2) \s_1 + \s_3 (\s_1 \cdot \s_2) \right ),
\end{multline}
respectively.

\section{Sensitivity Analysis}
\label{ap:B}
Applying $\DJ_{\y}^{\mi{\gamma}}$, where $|\mi{\gamma}|_1=k$ to 
\begin{equation}
\curl \left( \h \left(\y,\curl \Ar \left(\y \right) \right) \right) = \Jr,
\end{equation}
we obtain 
\begin{equation}
\curl \sum_{0 \leq \mi{\alpha}  \leq \mi{\gamma}} \binom{\mi{\gamma}}{\mi{\alpha}}\left(\DJ_\y^{\mi{\gamma}- \mi{\alpha}} \h_{\mi{\alpha}} \left(\cdot,\curl \Ar \left(\y\right) \right) \right) = 0.
\label{eq:ap2}
\end{equation}
Using Fa\`a di Bruno's formula we expand 
\begin{multline}
\DJ_\y^{\mi{\beta}} \h_{\mi{\alpha}} \left(\cdot,\curl \Ar \left(\y\right) \right) = \\
\sum_{\pi \in \Pi \left(\beta \right)} \Jac_\r^{\mathrm{ca}\left(\pi \right)}\h_{\mi{\alpha}} \left(\cdot,\curl \Ar \left(\y\right) \right)  \left(\curl \DJ_\y^{\pi_1} \Ar\left(\y\right),\dots, \curl \DJ_\y^{\pi_{\mathrm{ca}\left(\pi \right)}} \Ar\left(\y\right) \right),
\end{multline}
where $\Pi \left(\mi{\beta} \right)$ represents the set of partitions of $\beta$, $\mathrm{ca}\left(\pi \right)$ the cardinality of $\pi=\left \{\pi_1,\dots,\pi_{\mathrm{ca}\left(\pi \right)} \right \}$ and $\pi_i \subseteq \beta$ for $i=1,\dots,\mathrm{ca}\left(\pi \right)$, see \cite[Theorem 2]{clark2012faa}. Then, equation (\ref{eq:ap2}) can be rewritten as
\begin{multline}
\curl \left(\nud(\cdot,\curl \Ar) \curl \Ar_{\mi{\gamma}} \right) = \\ - \curl \sum_{0 \leq \mi{\alpha} \leq \mi{\gamma}} \binom{\mi{\gamma}}{\mi{\alpha}} 
  \left(\sum_{\pi \in \Pi^* \left(\gamma -\alpha \right)} \Jac_\r^{\mathrm{ca}\left(\pi \right)}\h_{\mi{\alpha}} \left(\cdot,\curl \Ar  \right)  \left(\curl  \Ar_{\mi{\pi}_1},\dots, \curl \Ar_{\mi{\pi}_{\mathrm{ca}(\pi)}} \right) \right),
\end{multline}
where $\Pi^* \left(\gamma -\alpha \right)$ is the defined as the set $\Pi \left(\gamma -\alpha \right)$, with $\mathrm{ca}\left(\pi \right) > 1$ if $\alpha=0$. We observe, that $\mi{\pi}_i < \mi{\gamma}$ for $i=1,\dots,\mathrm{ca}\left(\pi \right)$ and the summation carried out in the previous equation and hence the derivatives contained in the right--hand--side are of lower order.
\end{appendix}

\section*{Acknowledgment}
The authors would like to thank St\'ephane Cl\'enet for providing measurement data of $B$--$H$ curves and Herbert De Gersem for valuable comments and discussions on the subject.

\end{document}